\documentclass[graybox, envcountsame, envcountsect]{svmult}
\def\date{15.2.2014} 
\usepackage{amssymb, amsmath}
\usepackage{mathptmx}       
\usepackage{helvet}         
\usepackage{courier}        
\usepackage{type1cm}        
\usepackage{makeidx}         
\usepackage{graphicx}        
\usepackage{multicol}        
\usepackage[bottom]{footmisc}

\input{ownmacs.sty} 

\newcommand{\mlabel}{\label}

\renewcommand{\cV}{\mathcal{V}}

\newcommand{\cI}{\mathcal{I}}

\newcommand{\cL}{\mathcal{L}}
\newcommand{\cO}{\mathcal{O}}
\newcommand{\cH}{\mathcal{H}}

\newcommand{\Herm}{\mathop{{\rm Herm}}\nolimits}
\newcommand{\Pol}{\mathop{{\rm Pol}}\nolimits}

\renewcommand{\sp}{\mathop{{\rm sp}}\nolimits}
\renewcommand{\ham}{\mathop{{\rm ham}}\nolimits}

\newcommand{\indlim}{{\displaystyle \lim_{\longrightarrow}}\ }
\newcommand{\prolim}{{\displaystyle\lim_{\longleftarrow}}\ }

\newcommand{\matr}[1]{\begin{matrix}#1\end{matrix}}

\begin{document}

\title*{Weak Poisson structures on infinite dimensional manifolds and hamiltonian actions} 
\author{K.-H. Neeb, H. Sahlmann, T. Thiemann} 
\institute{K.-H. Neeb, 
Department of Mathematics, 
Friedrich--Alexander University Erlangen-Nuremberg, 
Cauerstrasse 11, 91058 Erlangen, Germany,  \\
Email: neeb@math.fau.de 
\and 
H. Sahlmann and T. Thiemann, 
Friedrich--Alexander University Erlangen-Nuremberg, 
Institute for Quantum Gravity, 
Theoretical Physics III, 
Staudtstr. 7, 91058 Erlangen, Germany,  
Email:thomas.thiemann@gravity.fau.de, hanno.sahlmann@gravity.fau.de} 
\maketitle

\centerline{\date}

\abstract{We introduce a notion of a weak Poisson structure on a 
manifold $M$ modeled on a locally convex space. This is done by specifying 
a Poisson bracket on a subalgebra $\cA \subeq C^\infty(M)$ which has to satisfy 
a non-degeneracy condition (the differentials of elements of $\cA$ separate 
tangent vectors) and we postulate the existence of smooth Hamiltonian vector 
fields. Motivated by applications to Hamiltonian actions, we focus 
on affine Poisson spaces which include in particular the linear and affine 
Poisson structures on duals of locally convex Lie algebras. 
As an interesting byproduct of our approach, we can associate to an invariant 
symmetric bilinear form $\kappa$ on a Lie algebra $\g$ and a $\kappa$-skew-symmetric 
derivation $D$ a weak 
affine Poisson structure on $\g$ itself. This leads naturally to a 
concept of a Hamiltonian $G$-action on a weak Poisson manifold 
with a $\g$-valued momentum map and hence to a generalization of 
quasi-hamiltonian group actions. }

\section{Introduction} \mlabel{sec:1}

In geometric mechanics symplectic and Poisson manifolds form the basic underlying 
geometric structures on manifolds. In the finite dimensional context, this provides 
a perfect setting to model systems whose states depend on finitely many parameters 
(\cite{MR99}). In the context of symplectic geometry, resp., 
Hamiltonian flows, Banach manifolds were introduced by Marsden (\cite{Mar67}), and 
Weinstein obtained a Darboux Theorem for strong symplectic Banach manifolds 
(\cite{Wei69}).\begin{footnote}{ 
A symplectic form $\omega$ on $M$ is called {\it strong} if, for every 
$p \in M$, every continuous linear functional on $T_p(M)$ is of the form 
$\omega_p(v,\cdot)$ for some $v \in T_p(M)$.}  
\end{footnote}
Schmid's monograph \cite{Sch87} provides a nice introduction to 
infinite dimensional Hamiltonian systems. 
For more recent results on Banach--Lie--Poisson spaces we refer to the recent work of 
Ratiu, Odzijewicz and Beltita (\cite{OdR03, OdR04}, \cite{BR05}, 
\cite{OdR08}, \cite{Ra11}) and in particular for \cite{Gl08} for certain classes of 
locally convex spaces. 

In the present note we describe a possible approach to Poisson structures on 
infinite dimensional manifolds that works naturally for smooth manifolds modeled 
on locally convex spaces, such as spaces of test functions, smooth sections 
of bundles and distributions (\cite{Ha82}, \cite{Ne06}). Our requirements are minimal 
in the sense that any other concept of an infinite dimensional 
Poisson manifold should at least satisfy our requirements. 

In the finite dimensional case, the main focus of the theory of Poisson manifolds 
lies on the Poisson tensor $\Lambda$ which is a section of the vector bundle 
$\Lambda^2(T(M))$ and defines a skew-symmetric form on each cotangent space 
$T^*_m(M)$. This does not generalize naturally to infinite dimensional manifolds 
because continuous bilinear maps may be of infinite rank. 
Our main point is to define a weak Poisson structure on a smooth manifold 
$M$ by a Poisson bracket $\{\cdot,\cdot\}$ 
on a unital subalgebra $\cA \subeq C^\infty(M)$ satisfying the Leibniz rule and the 
Jacobi identity. In addition to that, we require that $\cA$ is large in the sense 
that, for every $m \in M$, the differentials $\dd F(m)$, $F \in \cA$, 
separate the points in the tangent space $T_m(M)$. We also require for 
each $H \in \cA$ the existence of a smooth Hamiltonian 
vector field $X_H$ determined by 
$\{F,H\} = X_H F$ for every $F \in \cA$. 
The main difference to the traditional approaches is that we do not require 
the Poisson bracket to be defined on all smooth functions, instead we restrict
 the class of admissible differentials to define Poisson 
brackets. It turns out that this rather algebraic approach 
is strong enough to capture the main formal features 
of momentum maps and affine Poisson structures on locally 
convex space as well as their relations with Lie algebras and their duals. 
In the affine case $M = V$, the minimal choice of $\cA$ is the subalgebra generated by 
a point separating subspace $V_*$ of the topological dual space $V'$. In this context 
one can also enlarge the algebra $\cA$ by adding certain exponential functions and 
extend the Poisson bracket appropriately; see \cite{Wa13} for such constructions. 

Although our approach largely ignores geometric difficulties we hope that it provides 
a natural language for dealing with Poisson structures on rather general infinite dimensional 
manifolds and that this leads to precise specifications of the key difficulties arising for 
concrete examples. A discussion of similar structures is used in the 
context of hydrodynamics (\cite{Ko07}) and for free boundary problems 
(\cite{LMMT86}). 

One of our main objectives was to understand the nature of the 
affine Poisson structures arising implicitly on Lie algebras of smooth loops 
in the context of Hamiltonian actions 
of loop groups and quasihamiltonian actions \cite{AMM98} 
(Section~\ref{sec:6}).

Although the construction of the tangent bundle $T(M)$ of a locally convex manifold 
$M$ and the Lie algebra $\cV(M)$ of smooth vector fields on $M$ 
follows pretty much the constructions from finite dimensional geometry 
(cf.~\cite[Ch.~8]{HN11}), 
serious difficulties arise when one wants to put a smooth manifold 
structure on the cotangent bundle 
$T'(M) := \dot\cup_{p \in M} T_p(M)'$ whose elements are continuous linear functionals
on the tangent spaces $T_p(M)$ of $M$. This works well for Banach manifolds 
when the dual spaces carry the norm topology, but if $M$ is not modeled on a 
Banach space, there may not be any topology for which the natural chart changes 
for $T'(M)$ are smooth. Accordingly, cotangent bundles can be constructed 
naturally if $M$ is an open subset of a locally convex space 
or if the tangent bundle $T(M)$ is trivial, in which case  
$T(M) \cong M \times V$ leads to $T'(M) \cong M \times V'$, so that any locally 
convex topology on $V'$ leads to a manifold structure on $T'(M)$. This works 
in particular for Lie groups. 

Since our main concern is with the algebraic framework for Poisson structures, we 
do not go into analytical aspects of symplectic leaves which 
are already subtle for Poisson manifolds not modeled on Hilbert spaces 
(\cite{BR05, BRT07, Ra11}). 

The structure of this paper is as follows. In Section~\ref{sec:2} 
we introduce the notion of a weak Poisson manifold and discuss various 
types of examples, in particular affine ones and weak symplectic manifolds. 
We also take a brief look at Poisson maps arising from inclusions of 
submanifolds and from submersions. In Section~\ref{sec:3} we then 
turn to momentum maps, which we consider as Poisson morphisms into 
affine Poisson spaces which arise naturally as subspaces 
of the dual of a Lie algebra $\g$. If $\g$ is the Lie algebra of a 
Lie group, we also have a global structure coming from the corresponding 
coadjoint action, but unfortunately there need not be any locally convex 
topology on $\g'$ for which the coadjoint action is smooth. 

As an interesting byproduct of our approach, one can use an invariant 
symmetric bilinear form $\kappa$ and a $\kappa$-skew-symmetric 
derivation $D$ on a Lie algebra $\g$ to obtain a weak 
affine Poisson structure on $\g$ itself. This leads naturally to a 
concept of a Hamiltonian $G$-action on a weak Poisson manifold 
with a $\g$-valued momentum map. 
For the classical case where $G$ is the loop group $\cL(K) = C^\infty(\bS^1,K)$ 
of a compact Lie group and the derivation is given by the derivative, 
we thus obtain the affine action on $\g = \cL(\fk)$ which corresponds to the 
natural action of the gauge group
 $\cL(K)$ on gauge potentials on the trivial $K$-bundle over $\bS^1$. 
At this point we obtain a natural concept of a Hamiltonian $\cL(K)$-space 
generalizing the one used in the context of quasi-hamiltonian $K$-spaces, 
where it is only defined for weak symplectic manifolds (\cite{AMM98}, \cite{Me08}). 

\acknowledgement{We thank Helge Gl\"ockner, Stefand Waldmann and Anton Alekseev 
for discussions on the subject matter of this manuscript and for pointing out 
references.}

\section{Infinite dimensional Poisson manifolds} 
\mlabel{sec:2}

In this section we introduce the concept of a weak Poisson structure on a locally convex manifold. 
Our requirements are minimal in the sense that any other concept of an infinite dimensional 
Poisson manifold should at least satisfy our requirements. The concept 
discussed below is strong enough to capture the main algebraic features 
of momentum maps and the Poisson structure on the dual of a Lie algebra.

\subsection{Locally convex manifolds}

We first recall the basic concepts concerning 
infinite dimensional manifolds modeled on locally convex spaces. 
Throughout these notes all topological 
vector spaces are assumed to be Hausdorff.

Let $E$ and $F$ be locally convex spaces, $U
\subeq E$ open and $f \: U \to F$ a map. Then the {\it derivative
  of $f$ at $x$ in the direction $h$} is defined as 
$$ \dd f(x)(h) := (\partial_h f)(x) := \frac{d}{dt}\Big|_{t = 0} f(x + t h) 
= \lim_{t \to 0} \frac{1}{t}(f(x+th) -f(x)) $$
whenever it exists. The function $f$ is called {\it differentiable at
  $x$} if $\dd f(x)(h)$ exists for all $h \in E$. It is called {\it
  continuously differentiable}, if it is differentiable at all
points of $U$ and 
$$ \dd f \: U \times E \to F, \quad (x,h) \mapsto \dd f(x)(h) $$
is a continuous map. The map $f$ is called a {\it $C^k$-map}, $k \in \N \cup \{\infty\}$, 
if it is continuous, the iterated directional derivatives 
$$ \dd^{j}f(x)(h_1,\ldots, h_j)
:= (\partial_{h_j} \cdots \partial_{h_1}f)(x) $$
exist for all integers $1\leq j \leq k$, $x \in U$ and $h_1,\ldots, h_j \in E$, 
and all maps $\dd^j f \: U \times E^j \to F$ are continuous. 
As usual, $C^\infty$-maps are called {\it smooth}. 

Once the concept of a smooth function 
between open subsets of locally convex spaces is established, it is clear how to define 
a locally convex smooth manifold. The tangent bundle $T(M)$ and 
the Lie algebra $\cV(M)$ of smooth vector fields on $M$ 
are now defined as in the finite dimensional case (cf.~\cite[Ch.~8]{HN11}) 
and differential $p$-forms are defined as smooth functions 
on the $p$-fold Whitney sum $T(M)^{\oplus p}$. Although 
it is clear what the cotangent bundle is as a set, namely 
the disjoint union $T'(M) := \dot\cup_{p \in M} T_p(M)'$ of the topological 
dual spaces of the tangent spaces, in general it is not clear how to 
put a smooth manifold structure on $T'(M)$. This is due to the fact that 
the dual $V'$ of the model space $V$ need not carry a locally convex topology 
for which the chart changes for $T'(M)$ are smooth. 
For a Banach manifold this works with the natural Banach space structure 
on the dual, and it also works for manifolds with a single chart and 
the weak-$*$-topology on the dual, but for general locally convex manifolds $M$ 
there seems to be no natural smooth structure on $T'(M)$  
(see \cite{Ne06, Ha82} for more details). 

\subsection{Weak Poisson manifolds} 

\begin{definition} {\rm Let $M$ be a smooth manifold modeled on a locally convex space. 
A {\it weak Poisson structure} on $M$ is a  unital subalgebra 
$\cA \subeq C^\infty(M,\R)$, i.e., it contains the constant functions 
and is closed under pointwise multiplication, 
with the following properties: 
\begin{description}
\item[\rm(P1)] $\cA$ is endowed with a {\it Poisson bracket} $\{\cdot,\cdot\}$, 
this means that it is a Lie bracket, i.e., 
\begin{equation}
  \label{eq:jaco}
\{F,G\} = -\{G,F\}, \quad  
 \{F,\{G,H\}\} = \{ \{ F,G\},H\} + \{ G, \{ F, H\}\},
\tag{J} \end{equation}
and it satisfies the Leibniz rule 
\begin{equation}
  \label{eq:leibniz}
 \{F, G H\} = \{F, G\}H + G \{ F, H\}.\tag{L}
\end{equation}
\item[\rm(P2)] For every $m \in M$ and $v \in T_m(M)$ satisfying 
$\dd F(m)v = 0$ for every $F \in \cA$ we have $v = 0$. 
\item[\rm(P3)] For every $F \in \cA$, there exists a smooth vector field $X_H \in \cV(M)$ 
with $X_H F = \{F,H\}$ for $F,H \in \cA$. It is called the corresponding 
{\it Hamiltonian vector field}. 
\end{description}

If {\rm(P1-3)} are satisfied, then we 
call the triple $(M,\cA,\{\cdot,\cdot\})$ a {\it weak Poisson manifold}. 
}\end{definition}

\begin{remark} \mlabel{rem:3.2} 
(a) (P2) implies that the vector field $X_H$ in (P3) is uniquely determined 
by the relation $\{F,H\}(m) = (X_H F)(m) = \dd F(m) X_H(m)$ for every $F \in \cA$. 

(b) For $F,G,H \in \cA$, 
\[ [X_F, X_G] H = \{\{H,G\},F\} - \{\{H,F\}, G\} = \{H,\{G,F\}\} = X_{\{G,F\}}H,\] 
so that 
\begin{equation}
  \label{eq:brel1}
[X_F, X_G] = X_{\{G,F\}}\quad \mbox{ for } \quad F,G \in \cA.
\end{equation}
We also note that the Leibniz rule leads to 
\begin{equation}
  \label{eq:prodrel1}
X_{FG}  = F X_G + G X_F \quad \mbox{ for } \quad F,G \in \cA.
\end{equation}

(c) If $\{\cdot,\cdot\} \: \cA \times \cA \to \cA$ is a skew-symmetric bracket 
satisfying the Leibniz rule, then the Jacobiator 
\begin{align*}
J(F,G,H) 
&:= \{F,\{G,H\}\} + \{G,\{H,F\}\} + \{H,\{F,G\}\} \\
&= \{F,\{G,H\}\} - \{G,\{F,H\}\} - \{\{F,G\},H\} 
\end{align*}
defines an alternating map $\cA^3 \to \cA$ which satisfies the Leibniz rule 
in every argument. It vanishes if and only if $\{\cdot,\cdot\}$ is a Lie bracket, 
i.e., if (P1) is satisfied. 
For a subset  $\cS \subeq \cA$ generating $\cA$ as a unital algebra, 
this observation implies that $J$ vanishes if it vanishes for $F,G,H \in~\cS$. 

(d) If (P1) and (P2) are satisfied, then 
\eqref{eq:prodrel1} implies that the subspace of all elements 
$X \in \cA$ for which $X_H$ as in (P3) exists 
is a subalgebra with respect to the pointwise product. Therefore it suffices to 
verify (P3) for a generating subset $\cS \subeq \cA$. 
\end{remark}

\begin{remark} \mlabel{rem:2.3} 
From (P3) it follows that the value of the Poisson bracket 
\[ \{F,G\}(p) = \dd F(p) X_G(p) = - \dd G(p) X_F(p) \] 
in $p \in M$ only depends on $\dd F(p)$, resp., $\dd G(p)$. 
On the separating subspace 
\[ T_p(M)_* := \{ \dd F(p) \: F \in \cA \} \subeq T_p(M)' \] 
we thus obtain a well-defined skew-symmetric bilinear map 
\[ \Lambda_p \: T_p(M)_* \times T_p(M)_* \to \R, \quad 
\Lambda_p(\alpha,\beta) := \{ F,G\}(p) \quad \mbox{ for }\quad 
\alpha = \dd F(p), 
\beta= \dd G(p).\] 
This suggests an extension of the Poisson bracket 
to the subalgebra $\cB \subeq C^\infty(M)$ of those 
functions $F$, for which $\dd F(p) \in T_p(M)_*$ holds for every 
$p \in M$, by the formula 
\[ \{F,G\}(p) := \Lambda_p(\dd F(p), \dd G(p)).\] 
At this point it is not clear that this results in a smooth function 
$\{F,G\}$ nor that, for $G \in \cB$, there exists a smooth vector field 
$X_G$ on $M$ such that $\{F,G\} = X_G F$ holds for $F \in \cB$ 
(cf.\ Example~\ref{ex:2.12} below for criteria). 
If both these conditions are satisfied and, in addition, the Poisson bracket 
on $\cB$ satisfies the Jacobi identity, then we can also work with the 
larger algebra $\cB$ instead of $\cA$. 
\end{remark}

\begin{remark} Suppose that $M$ is a Banach manifold. 
The notion of a Banach--Poisson manifold used in 
\cite{OdR08, Ra11} differs from our concept of a weak Poisson structure 
on $M$ in the sense that it is required that $\cA = C^\infty(M)$ and that 
every continuous linear functional on the dual space $T_p(M)'$ of the form 
$\alpha^\sharp := \Lambda_p(\alpha, \cdot)\in T_p(M)''$ can be represented 
by an element of $T_p(M)$. 
\end{remark}

\begin{remark} \mlabel{rem:2.4} 
Let $(M,\cA,\{\cdot,\cdot\})$ be a weak Poisson manifold. For $p \in M$, 
we call 
\[ C_p(M) := \{ X_F(p) \: F \in \cA\} \subeq T_p(M) \] 
the {\it characteristic subspace in $p$}. 
Then 
\[ \omega_p \: C_p(M) \times C_p(M) \to \R, \quad 
\omega_p(X_F(p), X_G(p)) := \{F,G\}(p) 
= \dd F(p) X_G(p) = - \dd G(p) X_F(p) \] 
is a well-defined skew-symmetric form. On the Lie algebra 
\[ \ham(M,\cA) := \{ X_F \: F \in \cA\} \subeq \cV(M) \] 
of hamiltonian vector fields, every form $\omega_p$ defines a $2$-cocycle 
\[ \tilde\omega_p(X,Y) := \omega_p(X(p), Y(p)) \] 
because 
\[ \tilde\omega_p([X_F, X_G], X_H) =
\tilde\omega_p(X_{\{G,F\}}, X_H) = \{\{G,F\},H\}(p)\] 
and $\{\cdot,\cdot\}$ satisfies the Jacobi identity. 
\end{remark}

\subsection{Examples of weak Poisson manifolds} 

We now turn to natural examples of weak Poisson manifolds. 

\begin{example} (Finite dimensional Poisson manifolds) Every 
finite dimensional (paracompact) Poisson manifold $(M,\Lambda)$ carries a natural weak 
Poisson structure with $\cA := C^\infty(M)$ and 
$\{F,G\}(m) := \Lambda_m(\dd F(m), \dd G(m))$. Then 
$T_m(M)^* = \{ \dd F(m) \: F \in \cA\}$ implies (P2) and the existence of 
$X_H \in \cV(M)$ follows from the fact that every derivation of the algebra 
$C^\infty(M)$ is of the form $F \mapsto XF$ for some smooth vector field $X \in\cV(M)$ 
(\cite[Thm.~8.4.18]{HN11}). 
\end{example}

\begin{remark} Let $V$ be a real vector space. 
We call a linear subspace $V_* \subeq V^*$ {\it separating} if 
$\alpha(v) = 0$ for every $\alpha \in V_*$ implies $v = 0$. 
This implies that, for every finite dimensional subspace 
$F \subeq V$, the restriction map $V_* \to F^*$ is surjective, and this in 
turn implies that the natural map 
$S(V) \to \R^{V_*}$ of the symmetric algebra $S(V)$ over $V$ to the algebra of 
functions on $V_*$ is injective. 
\end{remark}

\begin{theorem} \mlabel{thm:3.4} {\rm(Affine Poisson structures)} 
Let $V$ be a locally convex space and 
$V_*\subeq V'$ be a separating subspace. Further, let 
\begin{itemize}
\item[\rm(a)] $\Lambda \: V_* \times V_* \to \R$ be a skew-symmetric 
bilinear map with the property that, for every $\alpha \in V_*$, there exists 
an element $\alpha^\sharp \in V$ with 
$\Lambda(\beta, \alpha)= \beta(\alpha^\sharp)$ for every $\beta \in V_*$, and 
\item[\rm(b)] let $[\cdot,\cdot]_0$ be a Lie bracket on $V_*$ for which 
  \begin{itemize}
  \item[\rm(i)] $\Lambda$ is a $2$-cocycle, i.e., 
$\Lambda([\alpha,\beta], \gamma) + \Lambda([\beta,\gamma], \alpha) 
+ \Lambda([\gamma,\alpha], \beta)=0$ for 
$\alpha,\beta, \gamma \in V_*$. 
  \item[\rm(ii)] The linear maps $\ad_0 \alpha \: V_* \to V_*, \beta \mapsto 
[\alpha,\beta]_0$ have continuous adjoint maps \break 
$\ad_0^*\alpha \: V \to V$ defined by 
$\beta (\ad_0^* \alpha v) = [\alpha,\beta]_0(v)$ for $\alpha, \beta \in V_*$ 
and $v \in V$. 
  \end{itemize}
This leads to a Lie algebra structure on 
the space $\hat V_* := \R 1 \oplus V_*$ of affine functions on $V$ by 
\[ [t+ \alpha, s + \beta] := \Lambda(\alpha,\beta) + [\alpha,\beta]_0
\quad \mbox{ for } \quad t,s \in \R, \alpha,\beta \in V_*.\] 
\end{itemize}
Let $\cA \cong S(V_*) \subeq C^\infty(V)$ denote the unital subalgebra 
generated by $V_*$. 
Then $\dd F(v) \in V_*$ for  $F \in \cA$ and $v \in V$, and 
\[ \{F, G\}(v) :=\la [\dd F(v), \dd G(v)], v \ra \quad \mbox{ for } \quad 
v \in V, F, G\in \cA\] 
defines a weak Poisson structure on $V$. 
\end{theorem}

This weak Poisson structure is {\it affine} in the sense that, for 
$\alpha,\beta \in V_*$, the function $\{\alpha,\beta\}$ on $V$ is affine.

\begin{proof} First we observe that, for every $F \in \cA$ and $v \in V$, 
the Leibniz rule implies that the differential $\dd F(v)$ is contained in $V_*$. 
Therefore $\{\cdot,\cdot\}$ defines a skew-symmetric bracket \break 
$\cA \times \cA \to \R^V$ satisfying the 
Leibniz rule. For $\alpha, \beta \in V_*$, the function $\{\alpha,\beta\}$ 
is contained in $\hat V_* \subeq \cA$, and this implies 
that $\{\cA,\cA\} \subeq \cA$. To verify  the Jacobi identity, it suffices to do 
this on the generating subspace $V_* \subeq \cA$ 
(Remark~\ref{rem:3.2}(c)). For $\alpha, \beta, \gamma \in V_*$ we 
have $\{\alpha,\{\beta,\gamma\}\} = [\alpha,[\beta,\gamma]]$, 
so that (P1) follows from the Jacobi identity in the Lie algebra $\hat V_*$. 
Condition (P2) follows from the fact that 
$V_*\subeq \cA$ separates the points of $V$. 
To verify (P3), we first observe that, for $\alpha \in V_*$ and $F \in \cA$, we have 
\[ \{F,\alpha\}(v) = \la [\dd F(v), \alpha], v \ra 
= \Lambda(\dd F(v), \alpha) + [\dd F(v), \alpha]_0(v) 
= \dd F(v)(\alpha^\sharp) -\dd F(v) (\ad_0 \alpha)^* v.\] 
Therefore the affine vector field 
\begin{equation}
  \label{eq:hamvec2}
X_\alpha(v) := \alpha^\sharp - (\ad_0 \alpha)^* v 
\end{equation}
is a smooth vector field satisfying (P3). Now 
(P3) follows from an easy induction and \eqref{eq:prodrel1} 
(cf.\ Remark~\ref{rem:3.2}(d)). 
This completes the proof. 
\smartqed\qed\end{proof}

Specializing to the two particular cases $[\cdot,\cdot]_0 = 0$ 
and $\Lambda = 0$, we obtain constant, resp., linear Poisson structures 
as special cases. 

\begin{corollary} \mlabel{cor:3.5} {\rm(Constant Poisson structures)} 
Let $V$ be a locally convex space, 
$V_*\subeq V'$ be a separating subspace and 
$\Lambda \: V_* \times V_* \to \R$ be a skew-symmetric 
bilinear map with the property that, for every $\alpha \in V_*$, there exists 
an element $\alpha^\sharp \in V$ with 
$\Lambda(\beta, \alpha)= \beta(\alpha^\sharp)$ for every $\beta \in V_*$. 
Let $\cA \subeq C^\infty(V)$ denote the unital subalgebra 
generated by the linear functions in $V_*$. 
Then 
\[ \{F, G\}(v) := \Lambda(\dd F(v), \dd G(v)) \quad \mbox{ for } \quad 
v \in V, F, G\in \cA\] 
defines a weak Poisson structure on $V$. 
\end{corollary} 

\begin{example} (Canonical Poisson structures) Let $V$ be a locally convex space and 
$V_* \subeq V'$ be a separating subspace, endowed with a locally convex 
topology for which the pairing $V_* \times V \to \R$ is separately continuous. 
We consider the product space 
$W := V \times V_*$. Then $W_* := V_* \times V$ is a separating subspace of 
$W' \cong V' \times (V_*)'$, 
\[ \Lambda((\alpha,v), (\alpha', v')) := 
\alpha(v') - \alpha'(v) \] 
is a skew-symmetric bilinear form on $W_*$, and 
for $(\alpha,v)^\sharp := (v,-\alpha) \in W$, we have
\[ \Lambda((\alpha,v), (\alpha', v')) 
= \la (\alpha,v), (v', -\alpha') \ra 
= \la (\alpha,v), (\alpha',v')^\sharp \ra.\] 
Therefore we obtain with Corollary~\ref{cor:3.5} on $W$ a constant 
weak Poisson structure with $\cA \cong S(W_*)$ which is given on $W_* \times W_*$ by 
$\Lambda$. 
\end{example}

\begin{corollary}  \mlabel{cor:3.6} {\rm(Linear Poisson structures)} 
Let $V$ be a locally convex space,
$V_*\subeq V'$ be a separating subspace and 
$[\cdot,\cdot]$ be a Lie bracket on $V_*$ for which 
the linear maps $\ad \alpha \: V_* \to V_*$ 
have continuous adjoint maps $\ad^*\alpha \: V \to V$. 
Let $\cA \subeq C^\infty(V)$ denote the unital subalgebra 
generated by $V_*$. 
Then 
\[ \{F, G\}(v) :=\la [\dd F(v), \dd G(v)], v \ra \quad \mbox{ for } \quad 
v \in V, F, G\in \cA\] 
defines a weak Poisson structure on $V$. 
\end{corollary}

For a version of the preceding corollary for Banach spaces, we refer to 
\cite[Thm.~3.2]{Ra11} and \cite{OdR03}. In this context $V$ is a Banach space 
and $V_* := V'$ is the dual Banach space. Typical examples of 
Banach--Lie--Poisson space are the duals of $C^*$-algebras and 
preduals of $W^*$-algebras. Here the example of the space 
$V = \Herm_1(\cH)$ of hermitian trace class operators on a Hilbert space $\cH$ 
is of particular importance in Quantum Mechanics. By the trace pairing, 
its dual can be identified with the Lie algebra of skew-hermitian compact operators. 

\begin{remark} In the context of Theorem~\ref{thm:3.4} one can enlarge the 
algebra $\cA \subeq C^\infty(V)$ under the following topological assumptions. 
We assume that $V_*$ carries a locally convex topology for which 
\begin{description}
\item[\rm(A1)] the pairing $\la \cdot, \cdot \ra \: V_* \times V \to \R$ is continuous, 
\item[\rm(A2)] the Lie bracket $[\cdot,\cdot] \:V_* \times V_* \to \hat V_*$ is continuous, 
\item[\rm(A3)] the map $V_* \times V \to V, (\alpha, v) \mapsto 
(\ad_0 \alpha)^* v$ is continuous, and 
\item[\rm(A4)] the map $\sharp \: V_* \to V$ is  continuous. 
\end{description}
Then 
\[ \cB := \{F  \in C^\infty(V) \: \dd F \in C^\infty(V,V_*)\} \] 
is a subalgebra of $C^\infty(V)$ with respect to the pointwise 
multiplication. For $F,G \in \cB$, the function 
\[ \{F,G\}(v) := 
[\dd F(v), \dd G(v)](v) 
= \la [\dd F(v), \dd G(v)]_0, v \ra 
+ \Lambda(\dd F(v), \dd G(v)) \] 
is smooth and so is the vector field 
\[ X_G(v) = -(\ad_0 \dd G(v))^* v + \dd G(v)^\sharp  \]  
on $V$ (cf.\ \eqref{eq:hamvec2}) which satisfies 
\[ \{F,G\} = X_G F= \la \dd F, X_G \ra 
\quad \mbox{ and } \quad \la \alpha, X_G(v) \ra = \la [\alpha, \dd G(v)], v \ra 
\quad \mbox{ for } \quad \alpha \in V_*, v\in V.\] 

For every $F \in \cB$, we now identify $\dd^2 F$ with a smooth 
function $\tilde\dd^2 F \: V \times V \to V_*$ which is linear in the second argument. 
The symmetry of the second derivative then leads to the relation 
\[ \dd^2 F_v(w,u) = \la \tilde\dd^2 F_v(w),u\ra = \la \tilde\dd^2 F_v(u),w\ra.\] 
We now show that $\{F,G\} \in \cB$. The calculation 
\begin{align*}
\dd\{F,G\}(v)(h)
&= [\dd^2 F_v(h), \dd G(v)](v) 
+ [\dd F(v), \dd^2 G_v(h)](v) 
+ \la [\dd F(v), \dd G(v)]_0, h \ra \\ 
&= \dd^2 F_v(h, X_G(v)) - \dd^2 G_v(h, X_F(v)) 
+ \la [\dd F(v), \dd G(v)]_0, h \ra \\ 
&= \la \tilde\dd^2 F_v(X_G(v)), h \ra 
- \la \tilde\dd^2 G_v(X_F(v)), h \ra 
+ \la [\dd F(v), \dd G(v)]_0, h \ra  
\end{align*}
shows that 
\[ \dd \{F,G\}(v) 
=  \tilde\dd^2 F_v(X_G(v)) 
- \tilde\dd^2 G_v(X_G(v)) 
+ [\dd F(v), \dd G(v)]_0 \] 
is a smooth $V_*$-valued function. Therefore the Poisson bracket 
extends to $\cB$. From 
\begin{align*}
\la [\dd F(v), \dd G(v)]_0, X_H(v)\ra 
&= \la [\dd F(v), \dd G(v)]_0, -(\ad_0 \dd H(v))^* v + \dd H(v)^\sharp\ra\\ 
&= \la \big[[\dd F(v), \dd G(v)]_0, \dd H(v)\big]_0, v \ra 
+ \Lambda([\dd F(v), \dd G(v)]_0, \dd H(v)) \\
&= \big[[\dd F(v), \dd G(v)]_0, \dd H(v)\big](v) 
= \big[[\dd F(v), \dd G(v)], \dd H(v)\big](v) 
\end{align*}
we now derive 
\begin{align*}
\{\{F,G\},H\}(v)
&= \dd^2 F_v(X_H(v), X_G(v)) -\dd^2 G_v(X_H(v), X_F(v)) 
+ \big[[\dd F(v), \dd G(v)], \dd H(v)\big](v).
\end{align*}
Now the symmetry of the second derivative implies that the Poisson 
bracket on $\cB$ satisfies the Jacobi identity, so that 
$(V,\cB,\{\cdot,\cdot\})$ also is a weak Poisson structure on $V$. 
\end{remark}

If $V$ is a Banach space with $V_* = V'$ (in particular if $\dim V < \infty$), 
then the preceding construction actually leads to all smooth functions 
$\cB = C^\infty(V)$, so that 
we are in the context of Banach--Lie--Poisson spaces. However, one can do better: 

\begin{remark} \mlabel{ex:2.12} (Gl\"ockner's locally convex Poisson vector spaces) 
To obtain Poisson structures on $V$ for the algebra 
$\cA = C^\infty(V)$ of all smooth functions, one has to impose stronger 
assumptions on topologies on $V$ and $V_*$. In \cite[Def.~16.35]{Gl08} these are 
encoded in the concept of a {\it locally convex Poisson vector space}, 
which requires that the locally convex space $V$ has the following properties: 
\begin{itemize}
\item[\rm(a)] For the topology of uniform convergence on compact ($S = c$), resp., bounded ($S= b$) 
subsets of $V$ (or even more general classes $S$ of subsets) the linear injection \break 
$\eta_V \: V \to (V'_S)'_S, \eta_V(v)(\alpha) = \alpha(v)$ is a topological embedding.  
\item[\rm(b)] The topology on every product space $V^n$ is determined by its restriction 
to compact subsets ($V$ is a $k^\infty$ space). 
\item[\rm(c)] The dual space $V_S'$ carries an $S$-hypocontinuous Lie bracket $[\cdot,\cdot]$, 
i.e., it is separately continuous and continuous on all subsets of the form 
$V_S' \times B$, $B \in S$.
\item[\rm(d)] The Lie bracket on $V_S'$ satisfies 
$\eta_V(v) \circ \ad_\alpha\in \eta_V(V)$ for $v \in V$ and $\alpha \in V_S'$. 
\end{itemize}
If these conditions are satisfied, then \cite[Thm.~16.40]{Gl08} asserts that, 
for two smooth functions $F,G \in C^\infty(V)$, their Poisson bracket 
\[ \{F,G\}(v) := \la [\dd F(v), \dd G(v)], v\ra \] 
is smooth and that 
\[ X_F(v) := -\eta_V^{-1}(\eta_V(v) \circ \ad(\dd F(v)))\] 
is a smooth vector field satisfying $\{G,F\} = X_F G$. 
As in the preceding remark it now follows that 
$(V,C^\infty(V),\{\cdot,\cdot\})$ is a weak Poisson manifold. 
This is the special case of Corollary~\ref{cor:3.6}, where $V_* = V_S'$. 
\end{remark}

\begin{example} \mlabel{ex:2.8} 
(a) Let $\g$ be a locally convex Lie algebra, i.e., a 
locally convex space with a continuous Lie bracket. 
We write $\g'$ for  its topological dual space, endowed with the 
weak-$*$-topology. Then Corollary~\ref{cor:3.6} applies to 
$V := \g'$ and  $V_* := \g$ because, for each $X \in \g$, 
the bracket map $\ad X \: \g \to \g$ has a continuous adjoint 
$\ad^* X \: \g' \to \g'$. 
If $\g$ is  finite dimensional, we thus obtain the 
KKS (Kirillov--Kostant--Souriau) Poisson structure on $\g^* = \g'$. 

(b) The preceding construction can be varied by changing 
the topology on $\g'$ and by passing to a smaller subspace. 
Let $\g_* \subeq \g'$ be a separating subspace 
on which the adjoint maps $\ad^*X\alpha := \alpha \circ \ad X$ induce for each 
$X \in \g$ a continuous linear map. Then Corollary~\ref{cor:3.6} applies with  
$V := \g_*$ and $V_* := \g$, and we thus obtain a weak Poisson structure on 
$\g_*$ for which the Hamiltonian functions $H_X(\alpha) = \alpha(X)$ satisfy 
\[ \{H_X, H_Y \} = H_{[X,Y]} \quad \mbox{ for } \quad X, Y \in g.\] 

(c) Suppose that $\g$ is a locally convex Lie algebra and 
$\kappa \: \g \times \g \to \R$ is a continuous non-degenerate 
symmetric bilinear form which is invariant under the adjoint representation, i.e., 
\[ \kappa([x,y], z) + \kappa(y,[x,z]) = 0 \quad \mbox{ for } \quad 
x,y,z \in \g.\] 
Then the natural map 
\[ \flat \: \g \to \g', \quad X^\flat(Y) := \kappa(X,Y) \] 
is injective and $\g$-equivariant with respect to the adjoint and coadjoint 
representation, respectively. We may thus apply (b) with 
$\g_* = \g^\flat= \{ X^\flat \: X \in \g\} \cong \g$ to obtain 
a linear weak Poisson structure 
on $\g$ with $\cA \cong S(\g)$. The 
Hamiltonian functions $X^\flat(Y) = \kappa(X,Y)$ satisfy 
\[ \{X^\flat, Y^\flat\} =[X,Y]^\flat \quad \mbox{ for } \quad X, Y \in \g.\] 

(d) Let $\g$ be a locally convex Lie algebra and 
$\omega \: \g \times \g \to \R$ be a continuous $2$-cocycle, i.e., 
\[ \omega([X,Y], Z) + \omega([Y,Z], X) + \omega([Z,X], Y) =0, \] 
so that $\hat\g = \R \oplus \g$ is a locally convex Lie algebra with 
respect to the Lie bracket 
\[ [(t,X), (s,Y)] := (\omega(X,Y), [X,Y]).\] 
We call it the {\it central extension defined by $\omega$}. 
Identifying the element $(t,X) \in \hat\g$ with the affine function 
$\alpha \mapsto t + \alpha(X)$ on $\g'$, we obtain with 
Theorem~\ref{thm:3.4} (for $V = \g'$ and $V_* = \g$) 
an affine weak Poisson structure on $\g'$, for which 
the Hamiltonian functions $H_X(\alpha) = \alpha(X)$, $X \in \g$, satisfy 
\[ \{H_X, H_Y \} = H_{[X,Y]} + \omega(X,Y) \quad \mbox{ for } \quad 
X,Y \in \g.\] 
The assumptions of Theorem~\ref{thm:3.4} are satisfied 
with $\Lambda = \omega$. 

More generally, suppose that $\g_* \subeq \g'$ is subspace separating the points of $\g$ 
and on which the adjoint maps $\ad^*X$, $X \in \g$, induce continuous 
endomorphisms. Assume further that it contains all functionals 
$i_X \omega$, $X \in \g$. Then Theorem~\ref{thm:3.4} 
yields an affine weak Poisson structure on $\g_*$ with 
\[ \{H_X, H_Y \} = H_{[X,Y]} + \omega(X,Y) \quad \mbox{ for } \quad 
X,Y \in \g.\] 

(e) To combine (c) and (d), we assume that, 
in addition to $\g$ and $\kappa$ as in (c), we are given 
a $\kappa$-skew symmetric continuous derivation 
$D \: \g \to\g$, so that $\omega(X,Y) 
= \kappa(DX,Y)$ is a $2$-cocycle. 
Then we obtain an affine  weak Poisson structure $(\cA,\{\cdot,\cdot\}_{\kappa,D})$ 
on $\g$ with $\cA \cong S(\g)$. The 
Hamiltonian functions $X^\flat(Y) := \kappa(X,Y)$ satisfy 
\[ \{X^\flat, Y^\flat\}_{\kappa,D} = [X,Y]^\flat + \kappa(DX,Y) \quad \mbox{ for } 
\quad X, Y \in \g.\] 
\end{example}

An important concrete class of examples to which the preceding constructions 
apply arise from loop algebras. We shall return to this example later, when 
we connect with Hamiltonian actions of loop groups 
(cf.\ Definition~\ref{def:4.3}). 

\begin{example} \mlabel{ex:2.9} Let 
$\fk$ be a Lie algebra which carries a non-degenerate 
invariant symmetric bilinear form 
$\la \cdot,\cdot \ra$. Then 
the {\it loop algebra of $\fk$} is the Lie algebra 
$\g := \cL(\fk) := C^\infty(\bS^1,\fk)$, endowed with the pointwise bracket. 
We identify the circle $\bS^1$ with $\R/\Z$ and, accordingly, 
elements of $\g$ with $1$-periodic functions on $\R$. 
Then $\kappa(\xi,\eta) = \int_0^1 \la \xi(t),\eta(t)\ra\, dt$ is a 
non-degenerate invariant symmetric bilinear form on $\g$ and 
$D\xi = \xi'$ is a skew-symmetric derivation. We thus obtain 
on $\g$ with Example~\ref{ex:2.8}(e) an affine 
weak Poisson structure with 
\[ \{\xi^\flat, \eta^\flat\} = [\xi,\eta]^\flat + \int_0^1 \la \xi'(t), \eta(t) \ra\, dt.\] 
\end{example}

\begin{remark} Typical predual spaces $\g_* \subeq \g'$ arise 
from geometric situations as follows 
(cf.\ \cite{KW09}): 

(a) If $\g = C^\infty(M,\fk)$, where $\fk$ is finite dimensional with a 
non-degenerate invariant symmetric bilinear form $\la \cdot,\cdot\ra$ 
and $\mu$ is a measure 
on $M$ which is equivalent to Lebesgue measure in charts, then we have 
an invariant pairing $\g \times \g \to \R, (\xi,\eta) \mapsto 
\int_M \la \xi,\eta\ra\, d\mu$ which leads to 
$\g_* \cong \g$. 

(b) If $M$ is a compact smooth manifold and 
$\g = \cV(M)$, the Fr\'echet--Lie algebra of smooth vector fields 
on $M$, then the space $\g_*$ of density-valued $1$-forms 
$\alpha$ on $M$ has a natural $\Diff(M)$-invariant pairing given by 
$(X,\alpha) \mapsto \int_M \alpha(X)$. Locally the elements of $\g_*$ 
are represented by smooth $1$-forms, so that $\g_*$ is much smaller than 
the dual space $\g'$ whose elements are locally represented by distributions. 
\end{remark}

In finite dimensions, symplectic manifolds provide the basic 
building blocks of Poisson manifolds because every Poisson manifold 
is naturally foliated by symplectic leaves. In the infinite dimensional 
context the situation becomes more complicated because a symplectic 
form $\omega \: V \times V \to \R$ on a locally convex space needs not 
represent every continuous linear functional on $V$. If it does, 
$\omega$ is called {\it strong}, and {\it weak} otherwise. Accordingly, 
a $2$-form $\omega$ on a smooth manifold $M$ is called {\it strong} if 
all forms $\omega_p$, $p \in M$, are strong, and {\it weak} otherwise. 

\begin{definition} \mlabel{def:sp-ham}
{\rm  A {\it weak symplectic manifold} is a pair $(M,\omega)$ of a smooth manifold 
$M$ and a closed non-degenerate $2$-form $\omega$. For a weak symplectic manifold we write 
\[ \ham(M,\omega) := \{ X \in\cV(M) \: (\exists H \in C^\infty(M))\, i_X \omega = \dd H\}\] 
for the Lie algebra of Hamiltonian vector fields on $M$ and 
\[ \sp(M,\omega) 
:= \{ X \in\cV(M) \: \cL_X \omega = \dd(i_X\omega) = 0\}\] 
for the larger Lie algebra of symplectic vector fields 
(cf.~\cite{NV10} for related constructions). 
}\end{definition}

\begin{proposition} \mlabel{prop:1.4} 
{\rm(Poisson structure on weak symplectic manifolds)} 
Let $(M,\omega)$ be a weak symplectic manifold. Then 
\[ \cA := \{ H \in C^\infty(M) \: (\exists X_H \in \cV(M)) \ \dd H = i_{X_H} \omega\}\]  
is a unital subalgebra of $C^\infty(M)$ and 
\[ \{ F,G \} := \omega(X_F, X_G) = \dd F(X_G) = X_G F\] 
defines on $\cA$ a Poisson bracket satisfying {\rm(P1)} and {\rm(P3)}. 

If, in addition,  for $v \in T_m(M)$, the condition 
$\omega(X(m),v) =0$ for every $X \in \ham(M,\omega)$, implies $v = 0$, 
then {\rm(P2)} is also satisfied.\begin{footnote}{This condition is satisfied for finite dimensional symplectic 
manifolds, for strongly symplectic smoothly paracompact 
Banach manifolds (cf.\ \cite{KM97}) and for symplectic vector spaces.}
\end{footnote} 
\end{proposition}

\begin{proof} Since $\omega$ is non-degenerate, the vector field $X_H$ is uniquely determined by~$H$. 
For $F, G \in \cA$ we have
\[ \dd(FG) = F \dd G + G \dd F = i_{F X_G + G X_F} \omega,\] 
which implies that $\cA$ is a unital subalgebra of $C^\infty(M)$. 

The closedness of the $1$-forms $i_{X_H} \omega$ implies that 
$\cL_{X_H} \omega = 0$. Further, $[\cL_X, i_Y] = i_{[X,Y]}$ leads to 
\begin{align*}
i_{[X_F, X_G]}\omega 
&= [\cL_{X_F}, i_{X_G}]\omega 
= \cL_{X_F}\big(i_{X_G}\omega\big)  = \cL_{X_F}\dd G \\
&=  \dd \big(i_{X_F}\dd G\big) + i_{X_F}\dd(\dd G)  =  \dd \big(i_{X_F}\dd G\big)  = \dd\{G,F\}. 
\end{align*}
Since $\omega$ is non-degenerate, this implies $\{\cA,\cA\} \subeq \cA$ with 
\begin{equation}
  \label{eq:iii}
[X_F, X_G] = X_{\{G,F\}} \quad \mbox{  for } \quad 
F,G \in \cA.   
\end{equation}

It is clear that $\{\cdot,\cdot\}$ is bilinear 
and skew-symmetric, and from $\dd(FG) = F \dd G + G \dd F$ we  conclude 
that it satisfies the Leibniz rule. So it remains to check the Jacobi
identity. This is an easy consequence of \eqref{eq:iii}: 
\begin{align*}
\big\{F,\{G,H\}\big\}
&=  X_{\{G,H\}} F
= - [X_G, X_H]F \\
&= - X_G(X_H F) + X_H(X_G F) =  \big\{G, \{F,H\}\big\} + \big\{\{F, G\},H\}.  
\end{align*}
We have thus verified (P1) and (P3). For (P2) we further need that, for every 
$v \in T_m(M)$, the condition that $\omega(X(m),v) =0$  for every $X \in \ham(M,\omega)$ 
implies $v = 0$.
\smartqed\qed\end{proof}

\begin{example} \mlabel{ex:symvec} If $(V,\omega)$ is a symplectic vector space, 
then a linear functional  $\alpha \: V \to \R$ is contained in the 
Poisson algebra $\cA$ if and only if there exists a vector 
$v \in V$ with $i_v \omega = \alpha$. 
Then $H_v = \alpha = i_v\omega$ is the Hamiltonian function of the constant vector 
field $v$. Accordingly, the Poisson structure on $V$ is determined by 
\begin{equation}
  \label{eq:brakrel2}
\{H_v, H_w \} = \dd H_v(w) = \omega(v,w)\quad \mbox{ for } \quad v,w \in V.  
\end{equation}
Here (P2) follows from the non-degeneracy of $\omega$. 
\end{example}

\subsection{Poisson maps} 

It is now clear how to define the notion of a Poisson map between two 
weak Poisson manifolds. Here we take a closer look at 
Poisson maps arising from inclusions of 
submanifolds and from submersions which correspond to regular Poisson reduction. 
In the context of Hamiltonian 
actions, Poisson maps to weak affine Poisson space arise as momentum maps. 

\begin{definition} 
{\rm Let $(M_j,\cA_j, \{\cdot,\cdot\}_j)$, $j =1,2$, be weak Poisson manifolds. 
A smooth map $\phi \: M_1 \to M_2$ is called a {\it Poisson map}, or {\it morphism of Poisson manifolds}, 
if $\phi^* \cA_{M_2} \subeq \cA_{M_1}$ and $\phi^*\{F,G\} = \{\phi^*F, \phi^*G\}$ 
for $F,G \in \cA_{M_2}$. 
}\end{definition}

\begin{proposition} {\rm(Poisson  submanifolds)} \mlabel{prop:1.3}
Let $(M,\cA,\{\cdot,\cdot\})$ be a weak Poisson manifold and 
$N \subeq M$ be a submanifold with the property that, for every $F \in \cA$, 
the restriction of $X_F$ to $N$ is tangential to~$N$. 
Then $\cI_N := \{ F \in \cA \: F\res_N = 0\}$ is an ideal with respect to the Poisson bracket, 
i.e., $\{\cI_N,\cA\} \subeq \cI_N$, 
and the induced bracket on $\cA_N := \cA/\cI_N \subeq C^\infty(N)$ 
defines a weak Poisson structure on $N$ 
such that the inclusion $N \into M$ is a morphism of weak Poisson manifolds. 
\end{proposition}

\begin{proof} First we show that $\cI_N$ is a Poisson ideal. 
So let $F \in \cI_N$ and $G \in \cA$. Then, for $n \in N$,  
$\{F,G\}(n) = \dd F(n) X_G(n) = 0$ because $F$ vanishes on $N$ and 
$X_G(n) \in T_n(N)$. This implies that $\cA_N$ inherits the structure of a Poisson algebra by 
\[ \{F\res_N, G\res_N\} := \{F,G\}\res_N,\] 
and that (P1) is satisfied. 

If $v \in T_n(N)$, $n \in N$, satisfies $\dd F(n)v = 0$ for every 
$F \in \cA_N$, then the same holds for $F \in \cA$, so that (P2) for $\cA$ implies 
(P2) for $\cA_N$. 

To verify (P3), we simply observe that our assumption implies that 
\[ \{F\res_N, G\res_N\} = \{F,G\}\res_N = (X_G F)\res_N = (X_G\res_N) F\res_N.\] 
\smartqed\qed\end{proof}

\begin{remark} \mlabel{rem:hyperplane}
(a) Let $\g$ be a locally convex Lie algebra and endow $\g'$ with the 
weak Poisson structure from Corollary~\ref{cor:3.6} above. Let 
$C \in \z(\g)$ be a central element. Then the hyperplane 
\[ N := \{ \alpha \in \g' \: \alpha(C) = 1 \} \] 
is a submanifold of $\g'$, and for every $F \in \cA_{\g'}$ and 
$\alpha \in N$, we have 
\[ 0 = X_F(\alpha) H_C = \la X_F(\alpha), C \ra,\] 
so that $X_F \in \cV(N)$. Therefore the assumptions of Proposition~\ref{prop:1.3} 
are satisfied, so that $\cA_N := \cA_{\g'}\res_N$ yields a weak Poisson structure 
on the hyperplane $N$. 

(b) The preceding restriction is of particular importance if 
we are dealing with a central extension 
$\tilde\g = \R \oplus_\omega\g$ of the Lie algebra $\g$ with the bracket 
\[ (z,X), (z',X') = (\omega(X,X'), [X,X']),\] 
where $\omega \: \g \times \g \to \R$ is a continuous $2$-cocycle. 
Then 
$C := (1,0)$ is a central element of $\tilde\g$ and 
\[ H_C^{-1}(1) = \{1\} \times \g' \subeq \tilde\g' \] 
inherits a Poisson structure from $\cA_{\tilde\g'}$. Identifying the affine space  
$\g'$ in the canonical fashion with the affine space $\{1\} \times \g'$, 
we thus obtain a weak Poisson structure on $\g'$, where 
$\cA \subeq C^\infty(\g')$ is generated by the continuous affine functions, i.e., 
$\cA \cong S(\g)$ as an associative algebra, and the Poisson bracket on $\cA$ is determined 
by 
\[  \{H_X, H_Y\} = H_{[X,Y]} + \omega(X,Y)  
\quad \mbox{ for } \quad H_X(\alpha) = \alpha(X), X \in \g, \alpha \in \g'\] 
(cf.\ Example~\ref{ex:2.8}(d)). 
\end{remark}

Let $q \: M \to N$ be a smooth {\it submersion}, i.e., 
$q$ is surjective and has smooth local sections. This implies in particular 
that the subalgebra $q^*C^\infty(N)$ consists of those smooth functions on 
$M$ which are constant along the fibers of~$q$. The following proposition discusses 
the most regular form of Poisson quotients. 

\begin{proposition}{\rm (Smooth Poisson quotients)} \mlabel{prop:2.17}
Let $(M,\cA_M,\{\cdot,\cdot\})$ be a  weak Poisson manifold 
and $q \: M \to N$ be a submersion. Then a Poisson subalgebra 
$\cB \subeq q^*C^\infty(N) \cap \cA_M$ 
is the image under $q^*$ of a weak Poisson structure on $N$ 
for which $q$ is a Poisson map if and only if 
\begin{equation}
  \label{eq:redcond}
\ker T_m(q) = \{ v \in T_m(M) \: (\forall F \in \cB)\, \dd F(m)v = 0\}.
\end{equation}
\end{proposition}

\begin{proof} Suppose first that 
$q$ is a Poisson map w.r.t.\ the weak Poisson structure 
$(\cA_N,\{\cdot,\cdot\})$ on~$N$. 
Then 
$\cB := q^*\cA_N \subeq \cA_M$ is a Poisson subalgebra 
and property (P2) of $\cA_N$ implies \eqref{eq:redcond}. 

Suppose, conversely, that $\cB \subeq q^*C^\infty(N) \cap \cA_M$ is a Poisson 
subalgebra satisfying \eqref{eq:redcond}.  
Let $\cA_N \subeq C^\infty(N)$ be the subalgebra with 
$q^*\cA_N = \cB$. Since $q^*$ is injective, $\cA_N$ inherits a natural 
Poisson algebra structure from $\cB$. Hence $(N,\cA_N,\{\cdot,\cdot\})$ 
satisfies (P1), and (P2) follows from \eqref{eq:redcond}. 
To see that (P3) also holds, let $f \in \cA_N$ and $F = q^*f \in \cB$.
Then the corresponding Hamiltonian vector field 
$X_F \in \cV(M)$ satisfies for every $G = q^*g\in \cB$ the relation 
\[ \dd g(q(m)) T_m(q) X_F(m)  
= \dd G(m) X_F(m) = \{G,F\}(m)= \{g,f\}(q(m)). \] 
For $m' \in M$ with $q(m) = q(m')$, this leads to 
\[ \dd g(q(m)) T_m(q) X_F(m)  = \dd g(q(m)) T_{m'}(q) X_F(m') \] 
for every $g$, so that (P2) implies 
$T_m(q) X_F(m)  = T_{m'}(q) X_F(m')$. Hence $X_F$ is projectable to a 
vector field $Y \in \cV(N)$ which is $q$-related to $X_F$. 
We then have for every $g \in \cA_N$ the relation 
$\{g,f\} = Y g$, so that (P3) is also satisfied. 
\smartqed\qed\end{proof}

\begin{remark} If, in the context of Proposition~\ref{prop:2.17}, 
the subalgebra $\cB$ is Poisson commutative, then (i) 
implies that the vector fields $X_F$, $F \in \cB$, are tangential 
to the fibers of $q$, hence projectable  to $0$. We thus obtain 
the trivial Poisson structure on $N$ for which all Poisson brackets 
vanish. 
\end{remark}

\section{Momentum maps} \mlabel{sec:3} 

We now turn to momentum maps, which we consider as Poisson morphisms to 
affine Poisson spaces which arise naturally as subspaces 
of the duals of Lie algebras $\g$. If $\g$ is the Lie algebra of a 
Lie group, we also have a global structure coming from the corresponding 
coadjoint action, but unfortunately there need not be any locally convex 
topology on $\g'$ for which the coadjoint action is smooth. 

\subsection{Momentum maps as Poisson morphisms} 

Since momentum maps are Poisson maps $\Phi \: M  \to V$, where 
$V$ carries an affine weak Poisson structure (Theorem~\ref{thm:3.4}), 
we start with a characterization of such maps. 

\begin{proposition} \mlabel{prop:3.1} 
Let $(V, \cA_V)$ be an affine Poisson manifold  
corresponding to a Lie algebra structure on the space 
$\hat\cA_* = \cA_* + \R 1$ of affine functions on $V$, 
$(M,\cA_M)$ a weak Poisson manifold and 
$\Phi \: M \to V$ a smooth map such that 
$\phi(\alpha) := \Phi^*\alpha = \alpha \circ \Phi \in \cA_M$ for every 
$\alpha \in \hat V_*$. Then the following are equivalent
\begin{itemize}
\item[\rm(i)] $\Phi^* \: \cA_M \to \cA_V$ is a homomorphism of Lie algebras, i.e., 
$\Phi$ is a Poisson map. 
\item[\rm(ii)] $\phi \: V_* \to \cA_M$ satisfies 
$\phi(\{\alpha,\beta\}) = \{ \phi(\alpha), \phi(\beta)\}$ for 
$\alpha,\beta \in V_*$. 
\item[\rm(iii)] $\Phi \: M \to V$ satisfies the equivariance condition 
  \begin{equation}
    \label{eq:equivrel}
T_m(\Phi) X_{\phi(\alpha)}(m) = X_\alpha(\Phi(m)) \quad \mbox{ for } \quad 
m \in M, \alpha \in V_*.
  \end{equation}
\end{itemize}
\end{proposition}

\begin{proof} (i) $\Rarrow$ (ii) is trivial. 

(ii) $\Rarrow$ (i): Clearly, $\Phi^* \: \cA_V \to \cA_M$ 
is a homomorphism of commutative algebras because 
$\Phi^*(V_*) \subeq \cA_M$ and $\cA_V$ is generated by $V_*$. 
Let $F,G \in \cA_V$. For $p \in M$ we put 
$\alpha := \dd F_{\Phi(p)}$ and $\beta := \dd G_{\Phi(p)}$, which are elements of $V_*$. 
Then 
\[ \dd(F \circ \Phi)_p = \alpha  \circ T_p(\Phi) 
= \dd(\Phi^*\alpha)_p = (\dd\phi(\alpha))_p,\] 
and we thus obtain 
\[ \{\Phi^*F, \Phi^* G\}(p) 
= \dd\phi(\alpha)_p X_{\Phi^*G}(p)
= \{ \phi(\alpha), \Phi^*G\}(p)
= \{ \phi(\alpha), \phi(\beta)\}(p)\] 
and 
\[ \phi([\alpha,\beta])(p) = \la [\alpha,\beta], \Phi(p)\ra 
= \{F,G\}(\Phi(p)). \] 
This proves that (ii) implies (i). 

(ii) $\Leftrightarrow$ (iii): The equivariance 
relation \eqref{eq:equivrel} is an identity for elements of $V$. 
Hence it is satisfied if and only if it holds as an identity of real 
numbers when we apply elements of the separating subspace $V_*$. This means that 
\[ \dd\phi(\beta)_m X_{\phi(\alpha)}(m) 
= \{\beta, \alpha\}(\Phi(m)) \quad \mbox{ for } \quad 
m \in M, \alpha,\beta \in V_*.\] 
Since the left hand side equals 
$\{\phi(\beta),\phi(\alpha)\}(m)$, this relation is equivalent to (ii). 
\smartqed\qed\end{proof}

The classical case of the preceding proposition is the one 
where $V = \g'$ is the dual of locally convex Lie algebra, endowed with 
the weak-$*$-topology. 

\begin{corollary} \mlabel{cor:3.2} 
Let $\g$ be a locally convex Lie algebra, endow 
$\g'$ with the canonical linear Poisson structure $\cA_{\g'}$, let 
$(M,\cA_M)$ be a weak Poisson manifold and 
$\Phi \: M \to \g'$ be a map such that all functions 
$\phi_X(m) := \Phi(m)(X)$ are contained in $\cA_M$. Then the following are equivalent
\begin{itemize}
\item[\rm(i)] $\Phi^* \: \cA_{\g'}\to \cA_M$ is a homomorphism of Lie algebras, i.e., 
$\Phi$ is a Poisson map. 
\item[\rm(ii)] $\phi \: \g \to \cA_M$ satisfies 
$\phi([X,Y]) = \{ \phi(X), \phi(Y)\}$ for 
$X,Y \in \g$. 
\item[\rm(iii)] $\Phi \: M \to \g'$ satisfies the equivariance condition 
  \begin{equation}
    \label{eq:equivrel2}
T_m(\Phi) X_{\phi(X)}(m) = - \Phi(m)\circ \ad X \quad \mbox{ for } \quad 
m \in M, X \in \g.
  \end{equation}
\end{itemize}
\end{corollary} 

\begin{remark} \mlabel{rem:3.3} If we endow $\g'$ with an affine 
Poisson structure corresponding to a Lie algebra cocycle 
$\omega$, then the condition Corollary~\ref{cor:3.2}(ii) has to be modified to 
\[ \{\phi(X),\phi(Y)\} = \phi([X,Y])  + \omega(X,Y) \quad \mbox{ for } \quad 
X,Y \in \g.\] 
\end{remark}

\begin{definition}\mlabel{def:momap}
{\rm An infinitesimal action of the locally convex Lie algebra 
$\g$ on the smooth manifold $M$ is a Lie algebra 
homomorphism $\beta \: \g \to \cV(M)$ 
for which all maps $\beta_m \: \g \to T_p(M), X \mapsto \beta(X)_m$ are continuous.  

If $(M,\cA_M,\{\cdot,\cdot\})$ is a weak Poisson manifold, then 
an infinitesimal action $\beta \: \g \to \cV(M)$ of a locally convex 
Lie algebra on $M$ is said to be {\it Hamiltonian} 
if there exists a homomorphism 
$\phi \: \g \to \cA_M$ of Lie algebras satisfying 
$X_{\phi(Y)} = -\beta(Y)$ for every $Y \in \g$. Then the map 
\[ \Phi \: M \to \g^*, \quad \Phi(m)(Y) := \phi_Y(m) \] 
is called the corresponding {\it momentum map}. 
Note that $\Phi(M) \subeq \g'$ is equivalent to the requirement that, for 
every $m \in M$, the linear functional $\g \to \R, Y \mapsto \phi_Y(m)$ is continuous. 
}\end{definition}

\begin{corollary} \mlabel{cor:A.IV.10}
 If $\Phi \: M \to \g'$ is a momentum map for a
Hamiltonian action of $\g$ on the weak Poisson manifold 
$(M,\cA_M,\{\cdot,\cdot\})$, 
then $\Phi$ is a Poisson map. 
\end{corollary}

\begin{example} For a locally convex Lie algebra $\g$, the infinitesimal 
coadjoint action 
$\beta \: \g\to \cV(\g')$ is given by the vector fields 
$\beta(X)(\alpha) := \alpha \circ \ad X = (\ad X)^*\alpha$. 
In view of Corollary~\ref{cor:3.2}, this action is Hamiltonian 
with momentum map $\Phi = \id_{\g'}$. 
\end{example}


\begin{remark} \mlabel{rem:2.5} (From symplectic actions to Hamiltonian actions) 
Let $(M,\omega)$ be a connected weak symplectic manifold 
and $\cA$ be as in Proposition~\ref{prop:1.4}. 
Further, let $\beta\: \g\to \sp(M,\omega)$ 
be an infinitesimal action by symplectic vector fields 
(cf.~Definition~\ref{def:sp-ham}). 
For $\beta$ to be a Hamiltonian action requires a lift of this homomorphism 
to a  Lie algebra homomorphism 
\[ \phi \: \g \to (\cA, \{\cdot,\cdot\}). \] 
A necessary condition for such a lift to exist is that 
$\beta(\g) \subeq \ham(M,\omega)$. 
Even if this is the case, 
such a lift does not always exist. To understand the obstructions, 
we recall the short exact sequence 
\[\0 \to  \R \to  \cA \to \ham(M,\omega) \to \0,\] which exhibits 
the Lie algebra $\cA$ as a central extension of the Lie algebra 
$\ham(M,\omega)$ 
(cf.~\cite{NV10} for an in depth discussion of related central extensions). 

Assuming that $\beta(\g) \subeq \ham(M,\omega)$, we consider the 
subspace 
\[ \hat\g := \{ (X,F) \in \g \oplus \cA \: \beta(X) = - X_F \} \] 
and observe that this is a Lie subalgebra of the direct 
sum $\g \oplus \cA$. Moreover, 
the projection $p(X,F) := X$ is a surjective homomorphism 
whose kernel consists of all pairs $(0,F)$, where $F$ is a constant function. 
We thus obtain the central extension 
\[ \R \cong \R(0,1) \to \hat\g\sssmapright{p} \g.\] 
The existence of a homomorphic lift $\phi \: \g \to \cA$ is equivalent 
to the existence of a splitting 
$\sigma \: \g \to \hat\g$. Therefore the obstruction to the existence 
of $\phi$ is a central $\R$-extension of $\g$, resp., a corresponding 
cohomology class in $H^2(\g,\R)$ (cf.\ \cite{Ne02}). 
\end{remark}

\subsection{Infinite dimensional Lie groups} 

Before we turn to momentum maps and Hamiltonian actions, we briefly 
recall the basic concepts underlying the notion of an infinite dimensional 
Lie group. 
A {\it (locally convex) Lie group} $G$ is a group equipped with a 
smooth manifold structure modeled on a locally convex space 
for which the group multiplication and the 
inversion are smooth maps. We write $\1 \in G$ for the identity element. 
Then each $x \in T_\1(G)$ corresponds to
a unique left invariant vector field $x_l$ with 
$x_l(\1) = x$. The space of left invariant vector fields is closed under the Lie
bracket of vector fields, hence inherits a Lie algebra structure. 
We thus obtain on $\g := T_\1(G)$ a continuous Lie bracket which
is uniquely determined by $[x,y] = [x_l, y_l](\1)$ for $x,y \in \g$. 
We shall also use the functorial notation $\L(G) := (\g,[\cdot,\cdot])$ 
for the Lie algebra of $G$ and, accordingly, 
$\L(\phi) = T_\1(\phi)\: \L(G_1) \to \L(G_2)$ 
for the Lie algebra homomorphism associated to 
a smooth homomorphism $\phi \: G_1 \to G_2$ of Lie groups. 
Then $\L$ defines a functor from the category 
of locally convex Lie groups to the category of locally convex 
Lie algebras. If $\g$ is a Fr\'echet, resp., a Banach space, then 
$G$ is called a {\it Fr\'echet-}, resp., a 
{\it Banach--Lie group}. 

A smooth map $\exp_G \: \L(G) \to G$  is called an {\it exponential function} 
if each curve $\gamma_x(t) := \exp_G(tx)$ is a one-parameter group 
with ${\gamma_x}'(0)= x$. 
Not every infinite dimensional Lie group has an exponential 
function (\cite[Ex.~II.5.5]{Ne06}), but exponential functions 
are unique whenever they exist. 

With the left and right multiplications 
$\lambda_g(h) := \rho_h(g) := gh$ we write 
$g.X = T_\1(\lambda_g)X$ and 
$X.g = T_\1(\rho_g)X$ for $g \in G$ and $X \in \g$. Then the two maps 
\begin{equation}
  \label{eq:tang-triv}
G \times \g \to TG, \quad (g,X) \mapsto g.X 
\quad \mbox{ and } \quad 
G \times \g \to TG, \quad (g,X) \mapsto X.g 
\end{equation}
trivialize the tangent bundle~$TG$. 

\subsection{Coadjoint actions and affine variants}

To add some global aspects to the Poisson structures on the dual 
$\g'$ of a Lie algebra $\g$, we assume that 
$\g = \L(G)$ for a Lie group $G$. Then 
the {\it adjoint action} of $G$ on $\g$ is defined by 
$\Ad(g) := \L(c_g)$, where $c_g(x) = gxg^{-1}$ is the conjugation map. The 
adjoint action is smooth in the sense that it 
defines a smooth map $G \times \g \to \g$. 
The {\it coadjoint action} 
on the topological dual space $\g'$ is defined by 
\[ \Ad^*(g)\alpha := \alpha \circ \Ad(g)^{-1}. \] 
The maps $\Ad^*(g)$ are continuous with respect to the 
weak-$*$-topology on $\g'$ and all orbit maps for $\Ad^*$ are smooth 
because, for every $X \in \g$ and $\alpha \in\g'$, 
the map $g \mapsto \alpha(\Ad(g)^{-1}X)$ is smooth. 
If $G$ is a Banach--Lie group, then the coadjoint action is smooth 
with respect to the norm topology on $\g'$, but in general it is not 
continuous, as the following example shows.
\begin{footnote}{By definition of the weak-$*$-topology on $\g'$, 
which corresponds to the subspace topology with respect to  the embedding 
$\g' \into \R^\g$, a map 
$\phi \:M \to \g'$ is smooth with respect to this topology 
if and only if all functions $\phi_X(m) := \phi(m)(X)$ are smooth on $M$. 
}\end{footnote}

\begin{example} \mlabel{ex:1.18} Let $V$ be a locally convex space and 
$\alpha_t(v) := e^t v$. Then the semidirect product 
\[ G := V \rtimes_\alpha \R, \qquad 
(v,t) (v',t') = (v + e^t v', t + t') \] 
is a Lie group. From $c_{(v,t)}(w,s) = ((1-e^s)v + e^t w, s)$ we derive that 
\[ \Ad(v,t)(w,s) = (e^t w - sv, s).\] 
Accordingly, we obtain 
\[ \Ad^*(v,t)(\alpha, u) = (e^{-t} \alpha, u + e^{-t} \alpha(v)).\] 
If $\Ad^*$ is continuous, restriction to $t = 1$ implies that 
the evaluation map 
\[ V' \times V \to \R, \quad (\alpha,v) \mapsto \alpha(v) \] 
is continuous, but w.r.t.\ the weak-$*$-topology on $V'$, this happens if 
and only if $V$ is finite dimensional. Therefore $\Ad^*$ is not continuous 
if $\dim V = \infty$.\begin{footnote}
{One can ask more generally, for which locally convex spaces 
$V$ and which topologies on $V'$ the evaluation map 
$V \times V' \to \R$ is continuous. This happens if and only if 
the topology on $V$ can be defined by a norm, and then the operator norm 
turns $V'$ into a Banach space for which the evaluation map is continuous. 
}\end{footnote}
\end{example}

\begin{remark} \mlabel{rem:gamma-D} 
(a) If $\g'$ is endowed with the affine Poisson structure 
corresponding to a $2$-cocycle $\omega \: \g \times \g \to \R$, 
then the corresponding infinitesimal action 
$\beta \: \g \to \cV(\g')$ of the Lie algebra $\g$ by affine vector 
fields need not integrate to an action of a connected 
Lie group $G$ with $\L(G) = \g$, but if $G$ is simply connected, 
then it does (cf.\ \cite[Prop.~7.6]{Ne02}). 

(b) The situation is much better for the Poisson structures 
on $\g$ discussed in Example~\ref{ex:2.8}(e). Then 
the Hamiltonian vector field associated to $Y \in \g$ is the affine 
vector field given by 
\begin{equation}
  \label{eq:aff-ad}
   X_{H_Y}(Z) = [Y,Z] - DY.
\end{equation}
Let $G$ be a Lie group with Lie algebra $\g$ and 
$\gamma_D \: G \to \g$ be a $1$-cocycle for the adjoint action with 
$T_\1(\gamma_D) = D$. Here the cocycle condition is 
\[ \gamma_D(g h) = \gamma_D(g) + \Ad_g\gamma_D(h) \quad \mbox{ for } \quad 
g,h \in G.\] 
Since the adjoint action is smooth, such a cocycle 
exists if $G$ is simply connected. Then we obtain an affine action of $G$ on $\g$ 
by 
\[ \Ad^D_g X := \Ad_g X - \gamma_D(g) \] 
integrating the given infinitesimal action of $\g$ determined by 
\eqref{eq:aff-ad}. 
\end{remark}

\begin{definition} 
  \mlabel{def:A.IV.6} 
{\rm Let $(M,\cA)$ be a weak Poisson manifold, $G$ a connected Lie group, and 
$\sigma \: G \times M \to M$ a smooth (left) action. We also write 
$g.p := \sigma_g(p) := \sigma^p(g) := \sigma(g,p)$ 
and define the vector fields 
\[ X_\sigma(p) := T_{(\1,p)}(\sigma)(X,0) \quad \mbox{ for } \quad X\in \g.\] 
Then we have a homomorphism 
\[  \L(\sigma) \: \g \to {\cal V}(M) \quad \hbox{ with } \quad 
X \mapsto -X_\sigma \] 
which defines an infinitesimal action of $\g$ on $M$. 

The action $\sigma$ is called {\it Hamiltonian} if its derived 
action $\L(\sigma)$ is Hamiltonian, i.e., if 
there exists a homomorphism of Lie
algebras $\phi \: \g \to \cA$ with $X_{\phi(Y)} = Y_\sigma$ 
for $Y \in \g$ such that, for every $m \in M$, the linear map 
$\Phi(m) \: \g \to \R, Y \mapsto \phi(Y)(m)$ is continuous. 
Then $\Phi \: M \to \g'$ is called the corresponding 
{\it momentum map} (cf.\ Definition~\ref{def:momap}). 
}\end{definition}

\begin{remark} For any smooth left action $\sigma \: G \times M \to M$ and 
$p \in M$, the right invariant vector field 
$X_r(g) = X.g$ on $G$ and the corresponding vector field $X_\sigma \in \cV(M)$ 
are $\sigma^p$-related. This follows from the relation 
$\sigma^p(hg) = h.\sigma^p(g)$ for $g,h \in G$. Combining this observation 
with the ``Related Vector Field Lemma'', one obtains a proof for 
$\L(\sigma) \: \g \to \cV(M)$ being a homomorphism of Lie algebras. 
\end{remark}

\begin{example} Let $(V,\omega)$ be a locally convex symplectic vector space 
and $G = (V,+)$ the translation group of $V$. 
Then the translation action $\sigma(v,w) := v+w$ 
of $V$ on itself is symplectic and every  constant vector field 
$v_\sigma(w) = v$ is Hamiltonian (cf.~Example~\ref{ex:symvec}). The relation 
\[ \{v,w\} =  \omega(v,w) \] 
shows that there is no homomorphism 
 $\phi \: \g \to \cA$ with 
$X_{\phi(v)} = v_\sigma$ for every $v \in \g$. 
\end{example}

\begin{remark} Of particular interest with respect to Poisson structures
are Lie groups $G$ whose Lie algebras $\g$ can be approximated in a natural 
way by finite dimensional ones. This can be done by direct or projective limits. 

(a) If $G = \indlim G_n$ is a Lie group whose Lie algebra 
$\g$ is a directed union of a sequence of 
finite dimensional subalgebras $\g_n = \L(G_n), n \in \N$,  
then $\g$ carries the finest locally convex topology which actually 
coincides with the direct limit topology  
(see \cite{Gl03, Gl05} for direct limit manifolds and Lie groups). 
Then its topological dual $V := \g'$, endowed with the topology 
of uniform convergence of bounded or compact subsets 
is a Fr\'echet space (isomorphic to a product $\R^\N$) and 
all assumptions (a)-(d) from Example~\ref{ex:2.12} are satisfied 
(\cite[Rem.~16.34]{Gl08}), so that 
we obtain a linear Poisson structure on $V = \g'_c = \g'_b$. In this case 
the coadjoint orbits of $G$ are unions of finite dimensional manifolds, 
which can be used to obtain symplectic manifold structures on them 
(cf.\ \cite{Gl03},\cite{CL13}). 

(b) The opposite situation is obtained for Lie groups $G = \prolim G_n$ 
which are projective limits of finite dimensional Lie groups $G_n$  (see \cite{HN09}). 
Typical examples are groups of infinite jets of diffeomorphisms. 
Here $\g$ is a Fr\'echet space (isomorphic to $\R^\N$) and 
the dual space $\g'$ is the union of the dual spaces $\g_n^*$. 
Endowed with the topology of uniform convergence of bounded or compact subsets 
the space $V = \g'$ satisfies all assumptions (a)-(d) from 
Example~\ref{ex:2.12} (\cite[Rem.~16.34]{Gl08}). In this case all coadjoint orbits are 
finite dimensional because they can be identified with coadjoint orbits 
of some $G_n$. 

In both cases we obtain weak Poisson structures 
on $\g'$ for which $\cA = C^\infty(\g')$ is the full algebra of smooth 
functions for a suitable topology which is the weak-$*$-topology in the first 
case and the finest locally convex topology in the second. 
\end{remark}

\begin{example} \mlabel{ex:ham-g*} Let $G$ be a Lie group and $\g= \L(G)$. Further,  
let $\g_* \subeq \g'$ be an $\Ad^*(G)$-invariant 
separating subspace endowed with a locally convex 
topology for which the coadjoint action $\Ad_*(g) := \Ad^*(g)\res_{\g_*}$ on $\g_*$ 
is smooth. Then $\g_*$ carries a natural linear weak Poisson structure 
with $\cA \cong S(\g)$  and 
\[ \{F,H\}(\alpha) = \la \alpha, [\dd F(\alpha), \dd H(\alpha)] \ra 
\quad \mbox{ for }\quad 
\alpha \in \g_*, F,H \in \cA\] 
(Example~\ref{ex:2.8}(b); see also \cite[Sect.~4.2]{Ra11} for similar 
requirements in the context of Banach spaces). 

For $X, Y \in \g$, we have $\{H_X, H_Y\} = H_{[X,Y]}$ and the corresponding 
Hamiltonian vector fields are $X_{H_Y}(\alpha) = -\alpha \circ \ad Y$. 
Therefore the coadjoint action $\Ad_*$ on $\g_*$ is Hamiltonian and its momentum map 
is the inclusion $\g_* \into \g'$. 

For the coadjoint action $\Ad_*$ of $G$ on $\g_*$, the ``tangent space'' to the orbit 
of $\alpha \in \g_*$ is the space 
$\{ X_{\Ad_*}(\alpha) \: X \in \g\} = 
\alpha \circ \ad(\g)$. This is also the characteristic subspace 
of the Poisson structure (cf.~Remark~\ref{rem:2.4}) and the corresponding 
skew-symmetric form is given by 
\[ \omega_\alpha(X_F(\alpha), X_H(\alpha)) = \{F,H\}(\alpha) 
= \dd F(\alpha) X_H(\alpha),\] 
resp., 
\[ \omega_\alpha(\alpha \circ \ad X, \alpha \circ \ad Y) 
= \{H_X,H_Y\}(\alpha) = H_{[X,Y]}(\alpha) = \alpha([X,Y]).\]  

Fix $\alpha \in \g_*$. Then we obtain on $G$ a $2$-form by 
\begin{align*}
\Omega_\alpha(X.g, Y.g) 
&:=
 \omega_{g.\alpha}(X_{\Ad_*}(g.\alpha),Y_{\Ad_*}(g.\alpha)) 
= \omega_{g.\alpha}([g.\alpha  \circ \ad X, g.\beta \circ \ad Y])\\
&= (g.\alpha)([X,Y]) = \alpha([\Ad_g^{-1}X, \Ad_g^{-1}Y]).
\end{align*}
This means that $\Omega$ is a left-invariant $2$-form on $G$. 
Since $(\Omega_{\alpha,\1})(X,Y) = \alpha([X,Y])$ is a $2$-cocycle, 
$\Omega$ is closed, the  radical of $\Omega_\1$ coincides with the 
Lie algebra of the stabilizer subgroup $G_{\alpha}$. 

If $\cO_\alpha :=\Ad_*(G)\alpha$ carries a manifold structure for which the 
orbit map $G \to \cO_\alpha$ 
is a submersion, we thus obtain on $\cO_\alpha$ the structure 
of  a weak symplectic manifold. However, if the Lie algebra 
$\g$ is not a Hilbert space, then it is not clear how to obtain a 
manifold structure on $\cO_\alpha$, resp., the homogeneous space 
$G/G_\alpha$. In any case, we may consider the pair 
$(G,\Omega_\alpha)$ as a non-reduced variant of the symplectic structure 
on the coadjoint orbit. 
\end{example}

\subsection{Cotangent bundles of Lie groups and their reduction} 

Let $G$ be a Lie group, $\g= \L(G)$ and 
$\g_* \subeq \g'$ be as in Example~\ref{ex:ham-g*}, 
so that the coadjoint action 
$\Ad_*$ on $\g_*$ is smooth. Then the ``cotangent bundle'' 
\[ T_*(G) := \bigcup_{g \in G} 
\{ \alpha \in T'_g(G) \: \alpha \circ T_\1(\rho_g) \in \g_*\}\] 
carries a natural Lie group structure for which it is isomorphic 
to the semidirect product $\g_* \rtimes_{\Ad_*} G$.  
Here we identify $(\alpha,g)$ with the element 
$\alpha \circ T_\1(\rho_g)^{-1} \in T_g(G)'$, which leads to an injection 
$T_*(G) \into T'(G)$. 

The lift of the left, resp., right  multiplications to $T_*(G)$ is given by 
\begin{equation}
  \label{eq:leftact}
 \sigma^l_g(\alpha,h) = (\alpha \circ \Ad_g^{-1}, gh) 
\quad \mbox{ and } \quad \sigma^r_g(\alpha,h) = (\alpha, hg).
\end{equation}
The corresponding infinitesimal action is given by the 
vector fields
\[ X_{\sigma^l}(\alpha,h) = (-\alpha \circ \ad X, X.h) \quad \mbox{ and } \quad 
X_{\sigma^r}(\alpha,h) = (0, h.X).\] 
The smooth $1$-form defined by 
\[ \Theta(\alpha, g)(\beta, X.g) := \alpha(X) \] 
is an analog of the Liouville $1$-form. It follows from \eqref{eq:leftact} that 
it is invariant
under both actions $\sigma^l$ and $\sigma^r$. Note that 
\[ \Theta(X_{\sigma^l})(\alpha,h) = \alpha(X) \quad \mbox{ and } \quad 
\Theta(X_{\sigma^r})(\alpha,h) = \alpha(\Ad_h X).\] 
Now 
\[ \Omega := -\dd \Theta \] 
is closed smooth $2$-form on $T_*(G)$. To see that it is non-degenerate,
we observe that its invariance under left and right translations and the 
Cartan formulas imply 
\begin{equation}
  \label{eq:ix}
(i_{X_{\sigma^l}} \Omega)_{(\alpha,g)}(\beta, Y.g) 
= \dd (i_{X_{\sigma^l}} \Theta)_{(\alpha,g)}(\beta, Y.g) = \beta(X) 
\end{equation}
and, for the constant vertical vector field $Z(\alpha,g) = \gamma \in \g_*$, the 
relation $\Theta(Z) = 0$ leads to 
\begin{equation}
  \label{eq:iz}
(i_Z \Omega)_{(\alpha,g)}(\beta, Y.g) 
= - (\cL_Z \Theta)_{(\alpha,g)}(\beta, Y.g) = \gamma(Y).
\end{equation}
We conclude that $(T_*(G), \Omega)$ is a weak symplectic manifold. 

We thus obtain by Proposition~\ref{prop:1.4} on 
$T_*(G)$ a weak Poisson structure on the subalgebra 
\[ \cA := \{ H \in C^\infty(T_*(G)) \: (\exists X_H \in \cV(T_*(G)) 
\ \dd H = i_{X_H} \Omega\}\subeq C^\infty(T_*(G)).\] 
Let $C^\infty_*(G)\subeq C^\infty(G)$ denote the subalgebra of smooth 
functions $H$ whose differential 
$\dd H$ defines a smooth section $G \to T_*(G)$, resp., a smooth function 
\[ \delta H\: G \to \g_*, \quad (\delta H)_g(X) := (\dd H)_g(X.g). \] 
Then  \eqref{eq:iz} shows that, for $H \in C^\infty_*(G)$, 
the vertical vector field on $T_*(G)$ defined by 
$X_H(\alpha,g) := (\delta H(g),0)$ satisfies 
\[ (i_{X_H} \Omega)_{(\alpha,g)}(\beta, Y.g) = (\delta H)_g(Y) 
= (\dd H)_g(Y.g).\] 
For the corresponding function $\tilde H$ on $T_*(G)$, we therefore 
have $\dd \tilde H = i_{X_H} \Omega$, so that $H \in \cA$. 
On the other hand, we have seen above that, for $X \in \g$, the 
function $H_X(\alpha, g) = \alpha(X)$ on $T_*(G)$ satisfies 
$\dd H_X = i_{X_{\sigma^l}} \Omega$. 
This shows that $\cA$ contains the subalgebra $C^\infty_*(G)$ and 
the algebra $S(\g)$ of polynomial functions on the first factor $\g_*$ 
generated by the functions $H_X$, $X \in \g$. We therefore have 
$S(\g) \otimes C^\infty_*(G) \subeq \cA$. 

The Poisson bracket vanishes on $C^\infty_*(G)$, and, for $X \in \g$ and 
$F \in C^\infty_*(G)$, we have 
\[ \{\tilde F, H_X\}(\alpha, g) = \dd F_g(X.g) = (X_r F)(g) 
= \tilde{X_r F}(\alpha,g).\] 
We also note that, for $X, Y \in \g$, we have by \eqref{eq:ix}
\begin{align*}
\{H_X, H_Y\}(\alpha,g) 
&= \Omega(X_{\sigma^l}, Y_{\sigma^l})(\alpha,g)
=  (i_{X_{\sigma^l}} \Omega)_{(\alpha,g)}(-\alpha \circ \ad Y, Y.g) \\
&= -(\alpha \circ \ad Y)(X) = \alpha([X,Y]) 
= H_{[X,Y]}(\alpha,g).
\end{align*}
This implies that $S(\g)$ and 
$\cB := S(\g) \otimes C^\infty_*(G)$ are Poisson subalgebras of $\cA$. 
In particular, $\cB$ defines a weak Poisson structure on $T_*(G)$. 

Consider the submersion $q \: T_*(G) \to \g_*, (\alpha,g) \mapsto \alpha$. 
Then $\cB \cap q^*C^\infty(\g_*) \cong  S(\g)$, and since 
$H_X(\alpha,g) = \la q(\alpha,g), X \ra$, condition \eqref{eq:redcond} 
in Proposition~\ref{prop:2.17} is satisfied. Therefore 
$q$ is a Poisson map if we endow $\g_*$ with the Poisson structure 
determined on $\cA_{\g_*} = S(\g)$ by 
$\{H_X,H_Y\} = H_{[X,Y]}$ for $X,Y \in \g.$

The fibers of $q$ are the orbits of the right translation action $\sigma^r$ 
which is a Hamiltonian action of $G$ on $T_*(G)$ and 
$\cB^{\sigma^r(G)} \cong S(\g)$ is the subalgebra of invariant functions in $\cB$. 
On the other hand, $q$ is a momentum map for the left action 
$\sigma^l$ of $G$ on $T_*(G)$. Therefore the passage to the orbit space 
$\g_* \cong T_*(G)/\sigma_r(G)$ is an example of Poisson reduction 
from the Hamiltonian action $\sigma^l$ to the coadjoint action 
$\Ad_*$ on $\g_*$ (cf.~\cite[Thm.~13.1.1]{MR99} for the finite dimensional case). 

\begin{remark} (Magnetic cotangent bundles) 
A natural variation of this construction is obtained by 
using a continuous $2$-cocycle $b \: \g \times \g \to \R$ 
to get a closed right invariant $2$-form $B \in \Omega^2(G)$. 
If $\pi \: T_*G \to G$ is the bundle projection, then 
\[ \Omega_b := \Omega + \pi^*B \] 
is a closed right invariant $2$-form on $T_*(G)$. Since its values in 
vertical directions are the same as for $\Omega$, the form $\Omega_b$ is also 
non-degenerate. We thus obtain an infinite dimensional version 
of a magnetic cotangent bundle (cf.~\cite[\S 6.6]{MR99}, \cite[\S 7.2]{MMOPR07}). 

The Poisson bracket on $C^\infty_*(G)$ still vanishes, and, for $X \in \g$ and 
$F \in C^\infty_*(G)$, we still have $\{\tilde F, H_X\} = \tilde{X_r F}$. 
But for $X, Y \in \g$ we obtain
\begin{align*}
\{H_X, H_Y\} 
&= \Omega(X_{\sigma^l}, Y_{\sigma^l})+ B(X_r, Y_r) 
= H_{[X,Y]}+ b(X,Y).
\end{align*}
Therefore the quotient Poisson structure on $\g_* \cong T_*(G)/\sigma^r(G)$ 
is the affine Poisson structure from 
Example~\ref{ex:2.8}(b) (see \cite{GBR08} for applications of these techniques). 
\end{remark}

\section{Lie algebra-valued momentum maps} 
\mlabel{sec:6}

We have already seen in Example~\ref{ex:2.8}(e) how to obtain from an invariant 
symmetric bilinear form $\kappa$ and a $\kappa$-skew-symmetric derivation 
$D$ a weak affine Poisson structures on a Lie algebra~$\g$. 
This leads naturally to a 
concept of a Hamiltonian $G$-action with a $\g$-valued momentum map. 
For the classical case where $G$ is the loop group $\cL(K) = C^\infty(\bS^1,K)$ 
of a compact Lie group and the derivation is given by the derivative, 
we thus obtain the affine action on $\g = \cL(\fk)$ which corresponds to the 
action of $\cL(K)$ on gauge potentials on the trivial $K$-bundle over $\bS^1$. 

\subsection{Hamiltonian actions for affine Poisson structures 
on Lie algebras} \mlabel{subsec:4.1}

Let $G$ be a Lie group with Lie algebra $\g$,  
$\kappa \: \g \times \g \to \R$ be a continuous $\Ad(G)$-invariant 
non-degenerate symmetric bilinear form and 
$D \: \g \to \g$ be a continuous derivation for which we have a smooth 
$\Ad$-cocycle $\gamma_D \: G \to \g$ with ${\gamma_D}'(\1) = D$. 
In Remark~\ref{rem:gamma-D}(b) we have seen that this leads to a smooth 
affine action of $G$ on $\g$ by 
\[ \Ad^D_g \xi := \Ad_g \xi - \gamma_D(g) \quad \mbox{ for } \quad 
g \in G, \xi \in \g. \] 
We recall from Example~\ref{ex:2.8}(e) that $\g$ carries a weak Poisson 
structure $\{\cdot,\cdot\} = \{\cdot,\cdot\}_{\kappa,D}$ 
with $\cA \cong S(\g)$, generated by the functions 
$\xi^\flat := \kappa(\xi,\cdot)$. It is determined by 
\[ \{\xi^\flat, \eta^\flat \} 
= [\xi,\eta]^\flat + \kappa(D\xi,\eta) \quad \mbox{ for } \quad \xi, \eta \in g.\] 
For any $F \in \cA$ and $\xi \in \g$, the linear functional 
$\dd F(\xi) \in \g'$ is represented by 
$\kappa$, hence can be identified with an element $\nabla F(\xi) \in \g$, 
the {\it $\kappa$-gradient of $F$ in $\xi$}. In these terms, the Poisson structure 
on $\g$ is given by 
\[ \{ F, H\}(\xi) :=\kappa(\xi, [\nabla F(\xi), \nabla H(\xi)]) 
+ \kappa(D \nabla F(\xi), \nabla H(\xi))\quad \mbox{ for } \quad 
F, H \in \cA, \xi \in \g.\] 
The corresponding Hamiltonian vector fields are determined by 
\begin{align*}
(X_H F)(\xi) 
&= \{F,H\}(\xi) = \kappa(\nabla F(\xi), [\nabla H(\xi), \xi]) 
+ \kappa(D \nabla F(\xi), \nabla H(\xi))\\
&= \dd F(\xi) \big([\nabla H(\xi),\xi] - D\nabla H(\xi)\big), 
\end{align*}
which leads to 
\[ X_H(\xi) = [\nabla H(\xi), \xi] - D\nabla H(\xi).\] 
For $H = \eta^\flat$, $\eta \in \g$, this specializes to
\begin{equation}
 \label{eq:eta}
 X_{\eta^\flat} = \ad \eta - D\eta = \eta_{\Ad^D} \quad \mbox{ for } 
\quad \eta \in \g.
\end{equation}

\subsection{Loop groups and the affine action on gauge potentials} 

An important example arises for 
$\bS^1 = \R/\Z$ and the loop group 
$G = \cL(K) := C^\infty(\bS^1,K)$ where $K$ is a Lie group 
for which $\fk$ carries a non-degenerate $\Ad(K)$-invariant 
symmetric bilinear form $\la \cdot,\cdot\ra$. 
We then put $D\xi = \xi'$ and 
$\kappa(\xi,\eta) = \int_0^1 \la \xi(t),\eta(t)\ra\, dt$ 
as in Example~\ref{ex:2.9}. 
Then $\gamma_D(g) = \delta^r(g) := g' g^{-1}$ is 
the right logarithmic derivative, so that 
\begin{equation}
  \label{eq:Ad^D}
\Ad_g^D \xi = \Ad_g \xi - g'g^{-1} =: \xi^g
\end{equation}
corresponds to the natural affine action on the space 
$\Omega^1(\bS^1,\fk) \cong C^\infty(\bS^1,\fk)$ of gauge potentials of the 
trivial $K$-bundle $\bS^1 \times K$ over~$\bS^1$. 

For $\xi  \in \cL(\fk)$, 
let $\gamma_\xi \: \R \to K$ denote the unique solution of the 
initial value problem 
\begin{equation}
  \label{eq:ivp}
 \gamma(0) = \1 \quad \mbox{ and } 
\quad \delta^l(\gamma) := \gamma^{-1}{\gamma}' = \xi.
\end{equation}
For each $s \in \R$ we write 
\[ \Hol_s \: \cL(\fk) \to K, \quad \xi \mapsto \gamma_\xi(s), \] 
for the corresponding {\it holonomy map}. It satisfies the 
equivariance relation 
\begin{equation}\label{eq:holeq} 
 \Hom_s(\xi^g) = g(0) \Hol_s(\xi) g(s)^{-1} \quad \mbox{ for } \quad g \in \cL(K).
\end{equation}
In particular $\Hol := \Hol_1$ is equivariant with respect to the conjugation 
action of $K$ on itself. This formula also implies that 
$\gamma_{\xi^g} = g(0) \gamma_\xi g^{-1}$, so that the affine $\cL(K)$-action on $\g$ 
corresponds on the level of curves to the multiplication 
with the pointwise inverse on the right. 

\begin{proposition} For any Lie group $K$ for which 
\eqref{eq:ivp} is solvable,\begin{footnote}{This is the case for so-called 
{\it regular Lie groups} (cf.~\cite{Ne06}). 
Banach--Lie groups and in particular finite dimensional Lie groups are regular.}  
\end{footnote}
the action \eqref{eq:Ad^D} of the subgroup 
$\Omega(K) := \{ g \in \cL(K) \: g(0) = \1 \}$ 
on $\g$ is free and its orbits coincide with the fibers of $\Hol$, so that 
$\Hol$ induces a bijection 
\[ \oline\Hol \: \cL(\fk)/\Omega(K) \to K, \quad 
[\xi] \mapsto \Hol_1(\xi).\] 
\end{proposition}

\begin{proof} The relation $\xi^g = \xi$ implies $\gamma_{\xi^g} = \gamma_\xi$, so that 
$g(0) \gamma_\xi = \gamma_\xi g$. 
For $g(0) = \1$ this implies that $g = \1$ is constant. 
Therefore the action of the subgroup 
$\Omega(K)$ on $\g$ is free and $\Hol$ is constant on the $\Omega(K)$-orbits. 

Suppose, conversely, that $\Hol(\xi) = \Hol(\eta)$, i.e., 
$g_1 := \gamma_\xi(1) = \gamma_\eta(1)$. 
Since $\xi$ and $\eta$ are periodic, $\gamma_\xi(t+1) = g_1 \gamma_\xi(t)$ and 
$\gamma_\eta(t+1) = g_1 \gamma_\eta(t)$ holds for all $t \in \R$. Therefore 
$g(t) := \gamma_\eta(t)^{-1} \gamma_\xi(t)$ is a smooth 
periodic curve defining an element of 
$\Omega(K)$ with $\gamma_\eta = \gamma_\xi g^{-1}$. This in turn leads to the relation 
\[ \eta = \delta^l(\gamma_\eta) = \delta^l(g^{-1}) + \Ad_g \delta^l(\gamma_\xi) 
= \Ad_g \xi - \delta^r(g) = \xi^g.\]
\smartqed\qed
\end{proof}

\begin{remark} (An attempt on Poisson reduction from $\cL(\fk)$ to $K$) 
For every connected Lie group $K$, the map 
$\Hol \: \cL(\fk) \to K$ is surjective and it is easy to see that it is a 
submersion. In view of Proposition~\ref{prop:2.17}, it makes sense to 
ask for a Poisson subalgebra $\cB \subeq \cA\cong S(\cL(\fk))$ 
that induces on $K$ a Poisson structure for which $q$ is a Poisson map. 
A natural candidate for $\cB$ is the invariant subalgebra  
\[ \cB := \cA^{\Omega(K)} \] 
consisting of $\Omega_K$-invariant functions in $\cA$, i.e., 
functions that are constant on the fibers of $\Hol$. 

If $K$ is compact, then the exponential 
function $\exp \: \fk \to K$ is surjective, and since 
$\Hol\res_{\fk} =\exp$, it follows that $\Hol(\fk) = K$, which in turn means that 
every $\Omega(K)$-orbit meets the subspace $\fk \subeq \cL(\fk)$ of constant 
functions. We conclude that the restriction map 
$R \: \cB \to \Pol(\fk)$ is injective and that its image consists 
of polynomial functions on $\cL(\fk)$ that are constant on the fibers of the exponential 
function. Let $T \subeq K$ be a maximal torus and $\ft = \L(T)$ be its 
Lie algebra. Then every $F \in \cB$ restricts to a polynomial 
$F\res_\ft$ which is constant on the cosets of the lattice 
$\ker(\exp\res_\ft)$, hence constant. Since every element 
$X \in \fk$ is contained in the Lie algebra of a maximal torus, it 
follows that $F$ is constant on $\fk$, and therefore $F$ is constant on 
$\cL(\fk)$. We conclude that $\cB = \R \1$ contains only constant functions. 

This shows that the algebra $\cA\cong S(\g)$ of polynomial functions is too small to 
lead to a sufficiently large algebra of $\Omega(K)$-invariant functions. 
It is an interesting question whether there exists a suitable 
larger Poisson algebra $\tilde\cA\supeq \cA$ for which 
$\tilde\cA^{\Omega(K)}$ satisfies the assumptions of 
Proposition~\ref{prop:2.17}. 
\end{remark}

\begin{definition} \mlabel{def:4.3} {\rm A {\it Hamiltonian $\cL(K)$-space}
    \begin{footnote}{This concept depends on the choice of the 
invariant symmetric bilinear form $\la \cdot,\cdot\ra$ on the Lie algebra $\fk$. 
Changing this form leads to a different Poisson structure on $\cL(\fk)$.}
\end{footnote}is a smooth 
weak Poisson manifold \linebreak $(M,\cA,\{\cdot,\cdot\})$, endowed with a 
smooth action $\sigma \: \cL(K)\times M \to M$ which has a smooth 
momentum map 
\[ \Phi \: M \to \cL(\fk) \] 
which is a Poisson map with respect to $(\cA,\{\cdot,\cdot\})$. 
\begin{footnote}
  {In \cite{AMM98} one finds this  concept for the special case 
where $(M,\omega)$ is a weak symplectic manifold. In this case one 
requires the action $\sigma$ to be symplectic and the existence of a 
smooth $\cL(K)$-equivariant map $\Phi \: M \to \cL(\fk)$ 
such that the functions 
\[ \phi(\xi)(m) := \kappa(\Phi(m), \xi) \quad \mbox{ satisfy }\quad 
i_{\xi_\sigma} \omega = \dd(\phi(\xi)).\] 
These conditions are easily verified to be equivalent to ours 
(cf.\ Proposition~\ref{prop:3.1}). 
}\end{footnote}
}\end{definition}

\begin{remark} Since the subgroup $\Omega(K) \subeq \cL(K)$ acts freely on $\g$ 
and $\Phi$ is equivariant, it also acts freely on $M$, so that we can consider
 the holonomy space 
\[ \Hol(M) := M/\Omega(K), \] 
and obtain a commutative diagram 
\[ \matr{ 
M & \mapright{\Phi} & \cL(\fk) \\ 
\mapdown{} & & \mapdown{\Hol} \\ 
\Hol(M) & \mapright{\oline\Phi} & K}\] 
\end{remark}

The geometric structure contained in the bottom row consists in an 
action of the Lie group $K \cong \cL(K)/\Omega(K)$ on the orbit space 
$\Hol(M)$ and an equivariant map $\oline\Phi \: \Hol(M) \to K$. 
If $(M,\omega)$ is weak symplectic, this 
is enriched by the data contained in natural differential forms 
on $\Hol(M)$ and $K$, which leads to the concept of a 
quasihamiltonian $K$-space for which 
$\oline\Phi \: M \to K$ plays the role of a {\it group-valued momentum map}. 
If $K$ is a compact Lie group and 
$\cL(K)$ denotes a suitable Banach--Lie group of differentiable loops, 
such as $H^1$-loops, then the 
Equivalence Theorem in \cite[Thm.~8.3]{AMM98} 
asserts that quasihamiltonian actions of $K$ are in one-to-one correspondence 
with Hamiltonian $\cL(K)$-actions on Banach manifolds $M$ for which the momentum map 
$\Phi \: M \to \cL(\fk)$ is proper. 

Since our setup for Hamiltonian $\cL(K)$-action uses only the invariant 
bilinear form on $\fk$, it is also valid for non-compact Lie groups $K$ 
and even for infinite-dimensional ones, provided $\fk$ carries an 
invariant non-degenerate symmetric bilinear form. 

In particular, the construction of a Lie group-valued momentum map 
$\mu = \exp \circ \Phi$ from a Lie algebra-valued momentum map 
$\Phi \: M \to \g$ 
with respect to a Poisson structure $\{\cdot,\cdot\}_{\kappa,D}$ on $\g$ 
(cf.\ \cite[Prop.~3.4]{AMM98}) works quite generally for any pair 
$(\kappa,D)$ as in Subsection~\ref{subsec:4.1},

\end{document}